\documentclass[final,onefignum,onetabnum,a4paper]{siamart190516}
\pdfoutput=1

\usepackage{lipsum}
\usepackage{mathrsfs}
\usepackage{amsmath,amssymb}
\usepackage{graphicx}
\usepackage{epstopdf}
\usepackage{algorithmic}
\usepackage[shortlabels]{enumitem}
\usepackage{amsopn}
\usepackage{multicol}
\usepackage{tikz,calc}
\usepackage{textgreek}

\ifpdf
  \DeclareGraphicsExtensions{.eps,.pdf,.png,.jpg}
\else
  \DeclareGraphicsExtensions{.eps}
\fi
\usepackage{subfigure}

\headers{Global minimization of integral functionals}{G. Fantuzzi \& F. Fuentes}

\title{Global minimization of polynomial integral functionals\thanks{Submitted to the editors DATE.
\funding{Part of GF's work was supported by an Imperial College Research Fellowship. FF's work was partially supported by NSF Grant No. 2012658 and the National Center for Artificial Intelligence CENIA FB210017, Basal ANID.}
}}

\author{
Giovanni Fantuzzi\thanks{Department of Mathematics, FAU Erlangen-N\"urnberg 
  (\email{giovanni.fantuzzi@fau.de}).}
\and
Federico Fuentes\thanks{Institute for Mathematical and Computational Engineering (IMC), School of Engineering and Faculty of Mathematics, Pontificia Universidad Cat\'olica de Chile 
  (\email{federico.fuentes@uc.cl}).}
}

\newsiamremark{remark}{Remark}
\newsiamremark{hypothesis}{Hypothesis}
\newsiamremark{example}{Example}
\newsiamthm{claim}{Claim}
\newsiamthm{assumption}{Assumption}

\crefname{hypothesis}{Hypothesis}{Hypotheses}
\crefname{assumption}{Assumption}{Assumptions}
\crefname{subsection}{section}{sections}

\AtBeginDocument{%
	\setlength{\oddsidemargin}{\dimexpr(\paperwidth-\textwidth)/2-1in}%
	\setlength{\evensidemargin}{\oddsidemargin}%
	\setlength{\topmargin}{%
		\dimexpr(\paperheight-\textheight)/2-\headheight-\headsep-1in}%
}

\DeclareMathOperator*{\argmax}{Argmax}

\newcommand{\bR}{\mathbb{R}}

\newcommand{\bF}{\mathbb{F}}

\newcommand{\bP}{\mathbb{P}}
\newcommand{\bN}{\mathbb{N}}
\newcommand{\calF}{\mathscr{F}}
\newcommand{\Nel}{N_{\rm el}}

\newcommand{\feSpaceW}{{U}_h}
\newcommand{\feSpaceU}{\feSpaceW^\beta}
 
\newcommand{\vv}{u}

\newcommand{\uu}{u}

\newcommand{\abs}[1]{\left\vert #1 \right\vert}
\newcommand{\dVolume}{\mathrm{d}x}
\newcommand{\mesh}{\mathcal{T}}
\newcommand{\riesz}[1]{\mathscr{L}_{#1}}

\ifpdf
\hypersetup{
  pdftitle={Global minimization of polynomial integral functionals},
  pdfauthor={G. Fantuzzi and F. Fuentes}
}
\fi

\begin{document}

\maketitle
\begin{abstract}
    We describe a `discretize-then-relax' strategy to \emph{globally} minimize integral functionals over functions $u$ in a Sobolev space subject to Dirichlet boundary conditions. The strategy applies whenever the integral functional depends polynomially on $u$ and its derivatives, even if it is nonconvex. The `discretize' step uses a bounded finite element scheme to approximate the integral minimization problem with a convergent hierarchy of polynomial optimization problems over a compact feasible set, indexed by the decreasing size $h$ of the finite element mesh. The `relax' step employs sparse moment-sum-of-squares relaxations to approximate each polynomial optimization problem with a hierarchy of convex semidefinite programs, indexed by an increasing relaxation order $\omega$. We prove that, as $\omega\to\infty$ and $h\to 0$, solutions of such semidefinite programs provide approximate minimizers that converge in a suitable sense (including in certain $L^p$ norms) to the global minimizer of the original integral functional if it is unique. We also report computational experiments showing that our numerical strategy works well even when technical conditions required by our theoretical analysis are not satisfied.
\end{abstract}



\section{Introduction}
\label{s:intro}

Minimizing an integral functional is a fundamental problem in material science, quantum mechanics, electromagnetism, optimal control, and any area of physics or engineering where energy minimization principles arise~\cite{Pedregal2000,Rindler2018}. A typical problem asks to minimize the functional
\begin{equation}\label{e:intro-problem}
	\calF(u) = \int_\Omega f\big(x,u(x),\nabla u(x)\big)\,\dVolume
\end{equation}
over functions $u$ that map a sufficiently regular (e.g., Lipschitz) bounded domain $\Omega\subseteq\bR^n$ into $\bR^m$.
One often requires $u$ to also satisfy a mixture of boundary conditions, integral constraints, pointwise inequalities, and partial differential equations~\cite{Antil2018,Hinze2009}. 
Here, we focus on problems where $u$ satisfies given Dirichlet boundary conditions, but is otherwise unconstrained, and for which minimizers exist; see \cref{s:setup} for details.

One way to approximate minimizers is to optimize $u$ over a finite-dimensional space using algorithms for finite-dimensional optimization.
Alternatively, one can discretize and solve the first-order optimality conditions satisfied by all sufficiently smooth stationary points of $\calF$, which for problems with only boundary constraints take the form of Euler--Lagrange equations~\cite{Hinze2009}. Doing so can be nontrivial in the presence of nonlinearities, but descent algorithms or (variants of) Newton's method often work well~\cite{Hinze2009}. One can also employ deflation~\cite{Farrell2015} or initial guess randomization to compute multiple stationary points, among which one hopes to find the desired minimizer.
All of these strategies, however, have a major shortcoming: unless $\calF$ is convex, one typically cannot guarantee that a \emph{global} minimizer has been found.

\begin{figure}
    \centering
    \includegraphics[width=0.9\linewidth,trim=0.5cm 0.3cm 0.5cm 0cm]{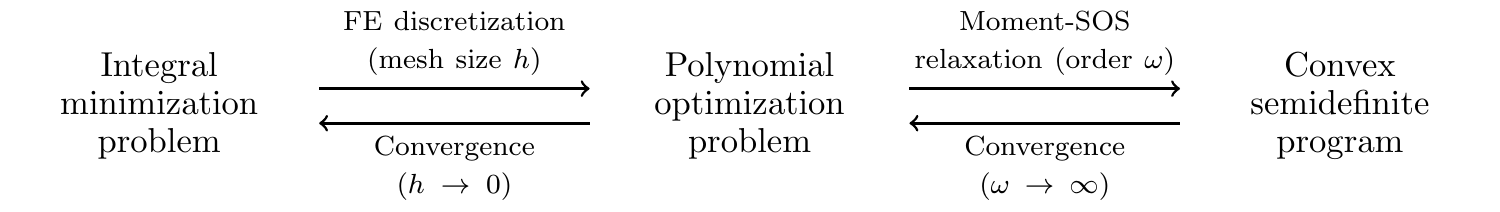}
    \caption{Summary of our discretize-then-relax strategy. In the diagram, $h$ is the size of the FE mesh in the `discretize' step, while $\omega$ indexes the moment-SOS hierarchy of SDPs in the `relax' step.\label{fig:approach-sketch}}
\end{figure}

This article describes a \emph{discretize-then-relax} approach that provably approximates global minimizers (called simply `minimizers' from now on) when $u$ is subject only to Dirichlet boundary conditions and $f$ depends polynomially on $u$ and its derivatives. \Cref{fig:approach-sketch} illustrates this approach. The `discretize' step uses finite elements (FE) to replace the minimization of $\calF$ with a polynomial optimization problem (POP) over a compact set.
The `relax' step employs \emph{moment-sum-of-squares} (moment-SOS) relaxations~\cite{Lasserre2001,Laurent2009,Parrilo2013,Lasserre2015,Lasserre2006,deKlerk2019} to approximate this POP with a hierarchy of computationally tractable semidefinite programs (SDPs), a well-known type of convex optimization problems~\cite{Vandenberghe1996,Ye1997,Nesterov2003,Nemirovski2006}. As one progresses through the hierarchy, optimal solutions of the SDPs converge to the moments of a probability measure supported on the set of minimizers of the POP. When this minimizer is unique, it can be read directly from the moment values. Otherwise, distinct minimizers can be recovered as long as the SDP solution satisfies an additional technical condition called the \textit{flat extension condition} (see~\cite[\S6.1.2]{Lasserre2015} and references therein). 
If minimizers of the POP recover minimizers of the integral functional $\calF$ as the FE mesh is refined, then one obtains a numerical scheme to solve~\cref{e:intro-problem} \emph{globally} with guaranteed convergence. 

The main contribution of this work is a rigorous convergence proof for this discretize-then-relax approach. A similar approach has already been used to search for solutions to finite-difference discretizations of differential equations \cite{Mevissen2008,Mevissen2009,Mevissen2011}, but without a proof of convergence. Here, beyond replacing finite differences with finite elements and differential equations with integral minimization problems, we establish the convergence of the overall numerical scheme by bootstrapping convergence results from polynomial optimization to the convergence of FE discretizations. To achieve this, we overcome two obstacles. The first is that standard FE discretizations of \cref{e:intro-problem} yield POPs over a full Euclidean space.
Moment-SOS hierarchies on noncompact domains recover the minimum of a POP if the optimal SDP solutions converge to an infinite sequence that grows no faster than a double exponential \cite[\S3.4.2]{lasserre2010book}. This condition, however, is usually impossible to check in practice. Moment-SOS hierarchies for POPs with compact feasible sets can instead be guaranteed \textit{a priori} to recover the global minimum of the POP using easily verifiable conditions.For this reason, in \cref{s:fe} we introduce and prove the convergence of a `bounded' FE discretization scheme that minimizes the functional $\calF$ in \cref{e:intro-problem} over mesh-dependent compact subsets of the usual conforming FE spaces. 
The second obstacle we must overcome is that one cannot guarantee that the POP minimizers can be extracted from the moment-SOS relaxations we use (see below for further discussion). We therefore further assume $\calF$ to have a unique global minimizer and show that this minimizer can be approximated arbitrarily accurately using functions constructed from the moment-SOS hierarchy (see \cref{th:main-result}). Moreover, the values of $\calF$ at these functions converge from above to the minimum of $\calF$ if the integrand $f$ satisfies further separability and convexity conditions (see \cref{th:conv-sos-ub}). Note that assuming $\calF$ to have a unique minimizer is by itself not enough to resolve the issue of minimizer extraction, because the POPs produced by our FE discretization scheme may still have multiple minimizers. Nevertheless, our analysis in \cref{s:convergence} shows that this potential difficulty can be bypassed.

Our discretize-then-relax strategy can be implemented in practice if the POPs obtained after discretizing \cref{e:intro-problem} have sparsely coupled optimization variables. For the boundary-constrained variational problems studied here, this is ensured by the local support of the basis functions used to build FE spaces. Exploiting sparsity can significantly reduce the computational complexity of moment-SOS relaxations~\cite{Waki2006,Lasserre2006,Wang2019,Wang2020,Wang2021tssos,Wang2021chordal-tssos,Zheng2021arc}, which would be prohibitively expensive for a densely-coupled POP with more than a few tens of variables. In \cref{s:sos}, we apply the so-called \textit{correlatively sparse} moment-SOS hierarchy from~\cite{Waki2006,Lasserre2006} to FE discretizations of problem \cref{e:intro-problem}, which can be done in general settings because it exploits couplings between POP variables that depend only on the FE mesh. (For particular integrands $f$ it may be possible to exploit further sparsity~\cite{Wang2019,Wang2020,Wang2021tssos,Wang2021chordal-tssos}, but here we focus only on the general case.)
The correlatively sparse moment-SOS hierarchy converges to the original POP if densely-coupled subsets of variables satisfy a so-called \textit{running intersection} property (RIP)~\cite{Lasserre2006,Grimm2007}. Nevertheless, two main limitations remain. First, even when the RIP holds, it is currently not known whether or not sparse moment-SOS relaxations generically satisfy the technical conditions required to extract multiple POP minimizers. It is for this reason that, as mentioned above, we assume the minimizer of the integral minimization problem \cref{e:intro-problem} to be unique. The second limitation is that while the RIP can always be ensured, doing so increases the computational cost of moment-SOS relaxations. This increase can be prohibitive when trying to solve integral minimization problems over two- or higher-dimensional domains, as shown by computational examples in \cref{s:examples}. Fortunately, convergence is observed in all of our examples even when the RIP is not enforced. 

In summary, we provide a provably convergent numerical method to globally solve integral minimization problems constrained by Dirichlet boundary conditions. We also report on several computational examples that, beyond confirming our theoretical analysis, showcase the practical convergence properties of our numerical strategy and provide insight into which theoretical assumptions could potentially be relaxed. At the level of generality considered here, our work appears to be the first to provide an explicit and provably convergent algorithm to solve nonconvex integral minimization problems to global optimality. Indeed, previous works either prove the convergence of FE discretizations without offering a practical algorithm to compute discrete minimzers \cite{Bartels2017sinum,Bartels2017cpa,Grekas2019,Bartels2020}, or use moment-SOS hierarchies to approximate discrete minimizers without analyzing convergence \cite{Mevissen2008,Mevissen2009,Mevissen2011}. As such, our work partially resolves a long-standing challenge in numerical analysis. Moreover, although we treat only boundary-constrained problems, we believe our analysis can be extended to more general settings, such as problems with well-posed differential equations as constraints. One could also replace conforming FE discretizations with nonconforming or discontinuous Galerkin (DG) methods, but we focus on the former to explain the main ideas behind our numerical approach in the simplest possible setting.

The rest of this article is organized as follows. \Cref{s:setup} introduces the class of integral minimization problems studied throughout the paper. In \cref{s:fe}, we define a bounded finite element discretization scheme and prove that the corresponding discretized problems converge to the original integral minimization problem as the mesh is refined. We describe sparsity-exploiting moment-SOS relaxations in \cref{s:sos}. In \cref{s:convergence} we prove the convergence of our `discretize-then-relax' approach. \Cref{s:examples} presents our computational experiments.  Final remarks and a summary of open challenges are offered in \cref{s:conclusion}.

\section{Mathematical setup}
\label{s:setup}

Let $\Omega\subset \bR^n$ be a bounded domain with Lipschitz boundary $\partial\Omega$. We assume $\Omega$ to be a polytope, but our analysis likely applies to domains with curved boundaries at the cost of additional technicalities (see, e.g., \cite[\S10.2]{BrennerScott}). As usual, $L^{p}(\Omega; \bR^m)$ for $p \in (1,\infty)$ is the Lebesgue space of functions from $\Omega$ into $\bR^m$ whose $p$-th power is integrable, $W^{1,p}(\Omega; \bR^m)$ stands for the Sobolev space of functions in $L^p(\Omega;\bR^m)$ with weak derivatives in $L^p(\Omega;\bR^m)$, and $W_0^{1,p}(\Omega; \bR^m)$ is the subspace of functions in $W^{1,p}(\Omega; \bR^m)$ that vanish on $\partial\Omega$. By the Rellich--Kondrachov theorem \cite[\S10.9]{Alt2016}, $W_0^{1,p}(\Omega; \bR^m)$ embeds compactly into $L^{q}(\Omega; \bR^m)$ for all $q < p^*$ with $p^* = \infty$ if $p\geq n$ and $p^* = np/(n-p)$ if $p<n$. This means a bounded sequence in $W_0^{1,p}(\Omega; \bR^m)$ has a subsequence converging in $L^{q}(\Omega; \bR^m)$.

We seek a numerical method to approximate the global minimizers of the generally nonconvex integral minimization problem
\begin{equation}\label{e:setup:min-problem}
    \calF^*
    := \min_{u \in W_0^{1,p}(\Omega; \bR^m)}
    \underbrace{\int_\Omega f\left(x,u(x),\nabla u(x)\right) \dVolume}_{=:\calF(u)}.
\end{equation}
We assume that the integrand $f:\Omega \times \bR^m \times \bR^{m \times n} \to \bR$ is measurable in its first argument and that $f(x,u,\nabla u)$ is polynomial in $u$ and $\nabla u$.
\begin{assumption}\label{ass:polynomial-f}
    The function $(y,z)\mapsto f(x,y,z)$ is polynomial for every $x\in\Omega$.
\end{assumption}
We also impose the following growth, coercivity and quasiconvexity conditions, which are standard in the calculus of variations. (Here and throughout the paper, $\abs{\,\cdot\,}$ stands for both the Euclidean norm of a vector and the Frobenius norm of a matrix.)
\begin{assumption}\label{ass:for-wlsc}
Let $p^* = \infty$ if $p\geq n$ and $p^* = np/(n-p)$ if $p<n$.
\begin{enumerate}[\textup{(}{A}1\textup{)},
                  leftmargin=*,
                  labelwidth=1ex,
                  topsep=1ex,
                  ]
    \item\label{ass:growth}
    There exists constants $B>0$ and $q<p^*$ such that $\abs{f(x,y,z)} \leq B (1 + \abs{y}^q + \abs{z}^p)$ for almost every $(x,y,z) \in \Omega \times \bR^m \times \bR^{m\times n}$.
    \item\label{ass:coercivity}
    There exists constants $\alpha_1>0$ and $\alpha_2 \geq 0$ such that $f(x,y,z) \geq \alpha_1\abs{z}^p - \alpha_2$ for almost every $(x,y) \in \Omega \times \bR^m$.
    \item\label{ass:quasiconvexity}
    The function $z \mapsto f(x,y,z)$ is quasiconvex for all $(x,y) \in \Omega \times \bR^m$, meaning
    \begin{equation*}
        f(x,y,z) \leq \inf_{\varphi \in W_0^{1,p}(\Omega; \bR^m)} \frac{1}{|\Omega|} \int_{\Omega} f\big(x,y,z+\nabla\varphi(\xi)\big)\, \mathrm{d}\xi.
    \end{equation*}
\end{enumerate}
\end{assumption}
It is well known \cite{Rindler2018,Dacorogna2008,Pedregal2000} that these assumptions ensure the existence of minimizers for \cref{e:setup:min-problem}. 
To see it, let us sketch the proof. The growth condition in \ref{ass:growth} ensures that $\calF(u)$ is finite on $W_0^{1,p}(\Omega; \bR^m)$. The coercivity condition in \ref{ass:coercivity}, instead, implies that any minimizing sequence $\{u_k\}$ is bounded in $\smash{W_0^{1,p}(\Omega; \bR^m)}$, so a subsequence (not relabelled) converges weakly in $\smash{W^{1,p}(\Omega; \bR^m)}$ to some element $u^* \in W_0^{1,p}(\Omega; \bR^m)$. Recall that a sequence $\{u_k\}$ is said to converge weakly to $u$ in $\smash{W^{1,p}(\Omega,\bR^m)}$, denoted by $u_k \rightharpoonup u$, if 
$\mathcal{L}(u_k)\to \mathcal{L}(u)$ for every continuous linear functional $\mathcal{L}$ on $W^{1,p}(\Omega;\bR^m)$.
Finally, the quasiconvexity of $f$ implies that $\calF(u)$ is (sequentially) weakly lower-semicontinuous on $\smash{W^{1,p}(\Omega; \bR^m)}$ \cite{Rindler2018}, so
$\calF^* = \liminf_{k\to \infty} \calF(u_k) \geq \calF(u^*) \geq \calF^*.$
All these inequalities must in fact be equalities, so $u^*$ minimizes $\calF$.

Finally, to establish the convergence of our numerical strategy, we assume that
\begin{assumption}\label{ass:uniqueness}
    The (global) minimizer of problem \cref{e:setup:min-problem} is unique.
\end{assumption}
Non-unique minimizers may occur, for instance, in problems with symmetries, in which case this assumption is a restriction. Note, however, that we still allow the minimization problem \cref{e:setup:min-problem} to have multiple stationary points beyond the global minimizer, including strictly suboptimal local minimizers.

We conclude with two remarks about our setup. 
First, it is straightforward to generalize our results to variational problems with the inhomogeneous boundary condition $u=g$, where $g \in W^{1-1/p,p}(\partial\Omega; \bR^m)$. Indeed, it suffices to extend $g$ to a function in $W^{1,p}(\Omega; \bR^m)$ and change variables to $v=u-g \in W_0^{1,p}(\Omega;\bR^m)$. Similar ideas apply to problems in which the Dirichlet condition $u=g$ is imposed only over a non-empty relatively open subset $\Gamma \subset \partial\Omega$, as one can change variables to $v=u-g$ in the space $W_{\Gamma}^{1,p}(\Omega;\bR^m)$ of Sobolev functions that vanish on $\Gamma$. One could even drop all boundary conditions if assumption~\ref{ass:coercivity} is strengthened into $f(x,y,z) \geq \alpha_1(\abs{y}^p+\abs{z}^p) - \alpha_2$.

Second, our approach and results can be generalized to problems with integrands $f(x,u,\nabla u, \ldots,\nabla^k u)$ that depend on all derivatives up to order $k$. In this case, one must require $u$ to be in the higher-order Sobolev space $W_0^{k,p}(\Omega;\bR^m)$, which embeds compactly into $W^{l,q}(\Omega;\bR^m)$ for all $l<k$ and all $q<p^*$, where $p^*=\smash{np/[n-(k-l)p]}$ if $(k-l)p<n$ and $p^*=\infty$ otherwise \cite{Alt2016}.
The growth condition~\ref{ass:growth} must be modified to require $\smash{f(\cdots,\nabla^l u, \cdots)}$ to grow no faster than $\smash{|\nabla^l u|^{q}}$ for some exponent $q$ that guarantees the compact embedding of $\smash{W_0^{k,p}(\Omega;\bR^m)}$ into $\smash{W^{l,q}(\Omega;\bR^m)}$.
The coercivity condition~\ref{ass:coercivity} and the quasiconvexity assumption \ref{ass:quasiconvexity} must also be modified to apply only to the last argument of $f$.
For scalar-valued problems ($m=1$) with $k=2$, one can alternatively replace assumption \ref{ass:quasiconvexity} with the 2-quasiconvexity condition
\begin{equation*}
    f(x,y,z,\lambda) \leq \inf_{\varphi \in W_0^{1,p}(\Omega)} \frac{1}{\abs{\Omega}} \int_\Omega f\left(x,y,z,\lambda+\nabla^2\varphi(x)\right) \dVolume
\end{equation*}
for every $\lambda \in \bR^{n\times n}$ and almost every $(x,y,z) \in \Omega \times \bR \times \bR^n$ \cite{DalMaso2004}.
These modifications to assumption \ref{ass:quasiconvexity} are needed to ensure the weak lower-semicontinuity of the integral functional $\int_\Omega f(x,u,\nabla u, \ldots,\nabla^k u)\,\dVolume$ in $W^{k,p}(\Omega;\bR^m)$.
\section{Finite element discretization}
\label{s:fe}

As stated in the introduction, the first step in our numerical approach to solving \cref{e:setup:min-problem} is to discretize it into a finite-dimensional POP over a compact set. We do this using a bounded FE scheme, which we introduce in \cref{ss:fe-space,ss:fe-bounded} for scalar-valued functions in $\smash{W_0^{k,p}(\Omega)}$. All definitions and results extend to vector-valued functions in $\smash{W_0^{k,p}(\Omega;\bR^m)}$ by discretizing each component separately. We will use this extension in \cref{ss:fe-pop} to approximate~\cref{e:setup:min-problem} by a discrete POP and prove that this approximation converges as the FE mesh is refined.

\subsection{Mesh, FE space, and degrees of freedom}\label{ss:fe-space}
Let $\{ \mesh_h \}_{h>0}$ be a sequence of meshes for the domain $\Omega$, indexed by a mesh size $h$ to be defined precisely below. Each mesh $\mesh_h$ consists of $\Nel$ open pairwise disjoint elements $T_1,\ldots,T_{\Nel} \subset \Omega$ such that $\overline{\Omega}=\cup_{e=1}^{\Nel} \overline{T_e}$.
The size of each element $T \in \mesh_h$ is $h_T=\max_{x,y\in T}|x-y|$, while $\rho_T$ is the the radius of the largest ball inscribed in $T$. The size of mesh $\mesh_h$ is $h=\max_{T\in\mesh_h} h_T$. We assume the family of meshes $\{\mesh_h\}_{h>0}$ is \emph{shape-regular}, meaning that there exists a constant $\sigma>0$, independent of $h$, such that $h_T/\rho_T < \sigma$ for every element $T$ of every mesh $\mesh_h$.
We also assume that every element $T \in \mesh_h$ is the image of a reference element $\hat{T}$ (independent of $h$) under an invertible affine map $A_T: \hat{T} \to T$.

Let $\ell\geq\max\{0,\lfloor k - n/p\rfloor\}$. For each mesh $\mesh_h$, let $\feSpaceW$ be the space of piecewise-polynomial functions $\vv \in C^\ell(\overline{\Omega})$ that vanish on $\partial\Omega$. Precisely, we set
\begin{equation}\label{e:Uh-definition}
    \feSpaceW := \big\{
    \vv \in C^\ell(\overline{\Omega}):\; \vv\vert_{\partial\Omega}=0 \text{ and }
    \vv\vert_{T} \circ A_T \in \mathcal{W}^{d}(\hat{T}) \quad \forall T \in \mathcal{T}_h
    \big\},
\end{equation}
where $\mathcal{W}^{d}(\hat{T})$ is a polynomial space of dimension $s$ that contains all $n$-variate polynomials of total degree $d$. This degree should be chosen such that $\feSpaceW \subset W^{k,p}(\Omega)$ and such that $ 1+d > \max\{\ell + n/p,k\}$ to ensure that $\feSpaceW$ possesses suitable approximation properties in $W^{k,p}(\Omega)$ as $h\to 0$ (see, e.g., \cite[Theorem~4.4.4]{BrennerScott} and \cite[Corollary~1.110]{ErnGuermond}). Note that $\feSpaceW$ has finite dimension, which we denote by $N$.

Given basis functions $\theta_1,\ldots,\theta_{s}$ for $\mathcal{W}^d(\hat{T})$, and setting $\theta_{T,i} = \theta_i \circ A_T^{-1}$, any $\vv \in \feSpaceW$ is represented locally on each element $T$ by
\begin{equation}\label{e:loc-dof-representation}
    \vv(x) = \sum_{i=1}^{s} \sigma_{T,i}(\vv) \, \theta_{T,i}(x) \quad\forall x\in T,
\end{equation}
where the $\sigma_{T,i}$ are bounded linear operators from $C^\ell(T)$ into $\bR$ called the \textit{local degrees of freedom} (DOF). 
One can show that there exists a constant $\gamma > 0$ such that
\begin{equation}\label{e:dof-growth-condition}
\vert \sigma_{T,i}(\vv) \vert \leq \gamma \|\vv\|_{C^\ell(T)}
\end{equation}
for all $i=1,\ldots,s$, all $T \in \mathcal{T}_h$, and all mesh sizes $h$. For instance, Lagrange finite elements have $\gamma=1$ because each DOF $\sigma_{T,i}(\vv)$ is a point evaluation of $\vv$.

It is also useful to introduce a global representation of functions in the FE space $\feSpaceW$. Since $\feSpaceW$ has finite dimension, $N$, there exists a family $\varphi_1,\ldots,\varphi_N$ of compactly supported global basis functions such that any $\vv \in \feSpaceW$ can be represented as
\begin{equation}\label{e:globalbasis}
    \vv(x) = \sum_{j=1}^N \xi_j \varphi_j(x)
\end{equation}
for some vector $\xi\in\bR^N$ of coefficients called the \emph{global DOF}. In particular,  each basis functions $\varphi_j$ can be chosen such that its restriction to an element $T$ of the mesh is either identically zero, or coincides with $\smash{\theta_{T,i}} = \smash{\theta_i\circ A_T^{-1}}$ for some basis function $\theta_i$ of  $\smash{\mathcal{W}^d(\hat{T})}$. As a result, given any ordering $T_1,\ldots,T_{\Nel}$ of the mesh elements, there exist ``global-to-local'' projection operators $\bF_1,\ldots,\bF_{\Nel}$ that extract from $\xi$ the local DOF needed to represent $\vv$ on the corresponding element:
\begin{equation}\label{e:global2local}
    \bF_e \xi =  \begin{pmatrix}\sigma_{T_e,1}(\vv) & \cdots & \sigma_{T_e,s}(\vv) \end{pmatrix} 
        \qquad \forall e=1,\ldots,\Nel.
\end{equation}
Finally, since each local DOF is a linear operator on $C^\ell(T)$ for at least one element $T$ on the mesh, we may view $\xi$ as a vector-valued linear operator from $C^\ell(\overline{\Omega})$ into $\bR^N$. We write $\xi(u)$ to highlight this interpretation in what follows.

\subsection{Interpolation with bounded FE sets}\label{ss:fe-bounded}
We now introduce compact subsets of $\feSpaceW$ that will be used in \cref{ss:fe-pop} to discretize the variational problem \cref{e:setup:min-problem} into a POP with a compact feasible set. 

Let $\beta:\bR_+ \to \bR_+$ be an increasing and unbounded function. The set of functions in $\feSpaceW$ whose global DOF $\xi_1,\ldots,\xi_N$ are bounded in magnitude by $\beta_h=\beta(h^{-1})$,
\begin{equation}\label{e:Uh-def-1d}
\feSpaceU := \left\{ \uu \in \feSpaceW:\quad
|\xi_j| \leq \beta(h^{-1}) \quad \, \forall j=1,\ldots,N
\right\},
\end{equation}
is clearly a compact subset of $\feSpaceW$.
We now show that it enjoys the same approximation properties in $W_0^{k,p}(\Omega)$ as the full FE space $\feSpaceW$.

\begin{proposition}\label{th:bounded-density}
	Let $\feSpaceW$ and $\feSpaceU$ be defined as in \cref{e:Uh-definition,e:Uh-def-1d}, where $\beta$ is a positive, increasing and unbounded function. For every $\uu \in W_0^{k,p}(\Omega)$, there exists a sequence $\{\uu_h\}_{h>0}$ with $\uu_h \in \feSpaceU$ such that $\|\uu_h-\uu\|_{W^{k,p}} \to 0$ as $h \to 0$.
\end{proposition}
\begin{proof}
    By a standard density argument, we may take $u$ to be smooth on $\overline{\Omega}$.
	Set $\beta_h=\beta(h^{-1})$ and let $\mathcal{P}\xi$ be the projection of $\xi\in\bR^N$ to the box $[-\beta_h,\beta_h]^N$. Define a ``saturated'' interpolation operator $\mathcal{I}_h: C^{\ell}(\overline{\Omega}) \to \feSpaceU$ via
	\begin{equation*}
	(\mathcal{I}_h \uu)(x) = \sum_{j=1}^{N} \big(\mathcal{P}\xi(u)\big)_j \,\varphi_j(x).
    \end{equation*}
	Now, since $\xi_j(u)$ coincides with some local DOF $\sigma_{T,i}(u)$, we find using \cref{e:dof-growth-condition} that
	\begin{equation*}
        0
        \leq \frac{\vert \xi_j(u) \vert}{\beta_h}
        = \frac{\vert \sigma_{T,i}(\uu) \vert}{\beta_h}
        \leq \frac{\gamma}{\beta_h} \, \|\uu\|_{C^\ell(T)}
        \leq \frac{\gamma}{\beta_h} \, \|\uu\|_{C^\ell(\overline{\Omega})}
	\end{equation*}
    for every $j=1,\ldots,N$. Since $\beta_h \to \infty$ as $h\to 0$, we conclude that there exists $h^* > 0$ such that $\vert \xi_j(u) \vert \leq  \beta_h$ for $j=1,\ldots,N$ when $h<h^*$. This means $\mathcal{P}\xi(u)=\xi(u)$ for all $h<h^*$, so  $\mathcal{I}_h$ reduces to the classical interpolation operator on the FE space $\feSpaceW$. Standard arguments \cite{Ciarlet2002,ErnGuermond,BrennerScott} then yield $\|\uu - \mathcal{I}_h \uu\|_{W^{k,p}} \to 0$ as $h \to 0$.
\end{proof}

\begin{remark}\label{rem:general-cnstr}
    \Cref{th:bounded-density} remains true if the constraints $|\xi_j| \leq \beta_h$ are replaced by the constraints $\bF_1 \xi,\ldots,\bF_{\Nel} \xi\in K_h$ for some compact set $K_h \subset \bR^s$ that contains the $s$-dimensional box $[-\beta_h,\beta_h]^s$. Indeed, 
    it suffices to redefine $\mathcal{P}\xi$ as the projection onto $\{\xi\in\bR^N: \bF_1 \xi,\ldots,\bF_{\Nel} \xi \in K_h\}$.
\end{remark}

\subsection{Discrete POP and convergence with mesh refinement}
\label{ss:fe-pop}
The FE construction of the previous sections can of course be generalized to vector-valued functions, leading to bounded FE sets $\feSpaceU \subset W^{1,p}_0(\Omega;\bR^m)$ that are `asymptotically dense' as $h\to 0$ in the sense of \cref{th:bounded-density}. Restricting the minimization in the variational problem \cref{e:setup:min-problem} to functions in $\uu_h$ leads to the finite-dimensional problem
\begin{equation}\label{e:discrete-problem}
    \calF_h^* := \min_{ \uu_h \in \feSpaceU} \int_\Omega f(x,\uu_h,\nabla \uu_h) \, \dVolume.
\end{equation}
We claim that this is a POP for the DOF $\xi\in \bR^N$ of functions in $\uu_h$, that its feasible set is compact, and that its minimizers recover those of \cref{e:setup:min-problem} as the mesh is refined.

The first two claims follow from the definition of $\smash{\feSpaceU}$. Indeed, writing $\uu_h$ as in~\cref{e:loc-dof-representation} on each mesh element, integrating in $x$, and recalling from \cref{e:global2local} that the vector $\bF_e\xi$ lists the $ms$ local DOF for element $T_e$ ($s$ DOF for each of the $m$ components of $\uu_h$),
\vspace{-\abovedisplayskip}
\begin{align}\label{e:unconstrained-objective-discretization}
\int_{\Omega} f(x, \uu_h, \nabla \uu_h) \dVolume
&=\sum_{e=1}^{\Nel} \int_{T_e} f(x, \uu_h, \nabla \uu_h) \dVolume 
\\ \nonumber
&=\sum_{e=1}^{\Nel} \underbrace{\int_{T_e} f\bigg(x, \sum_{i=1}^{s} \sigma_{T_e,i}(u_h) \theta_{T_e,i}, \sum_{i=1}^{s} \sigma_{T_e,i}(u_h) \nabla \theta_{T_e,i} \bigg) \dVolume}_{=: f_e(\bF_e \xi)}
\\[-3ex] \nonumber
&=\sum_{e = 1}^{\Nel} f_e(\bF_e\xi).
\end{align}
Each function $f_e:\bR^{ms} \to \bR$ is a polynomial because the integrand on the second line, viewed as a function of the local DOF, is the composition of a polynomial (cf. \cref{ass:polynomial-f}) with a linear function. Moreover, setting $\beta_h=\beta(h^{-1})$, the definition of $\feSpaceU$ requires $\xi_j^2 \leq \beta_h^2$ for every $j\in\{1,\ldots,N\}$, so
\begin{equation}\label{e:pop}
    \calF_h^* = \min_{\xi \in \bR^N} \left\{ \sum_{e = 1}^{\Nel} f_e(\bF_e\xi)\quad\text{subject to}\quad \beta_h^2 - \xi_j^2 \geq 0 \quad\forall j=1,\ldots,N \right\}.
\end{equation}

Any minimizer $\xi^*$ for \cref{e:pop} can be used to construct a minimizer $\uu_h^*$ for \cref{e:discrete-problem} via \cref{e:globalbasis}. We now show that these minimizers converge (up to subsequences) to minimizers of the integral minimization problem \cref{e:setup:min-problem} as the mesh size $h$ is reduced. We believe the proof to be standard, but include it for completeness. The argument may be viewed as an application of {\textGamma-convergence} \cite{DeGiorgi1975,Braides2014}, which offers a systematic framework to establish the convergence of FE discretizations (see, e.g., \cite{Bartels2017sinum,Bartels2017cpa,Bartels2020,Grekas2019}).

\begin{proposition}\label{th:unc-gamma-conv}
     Under \cref{ass:for-wlsc} and the conditions of \cref{th:bounded-density}, ${\calF^*_h \to \calF^*}$ as $h\to 0$. Moreover, every sequence ${\{\uu_h^*\}_{h>0}}$ of minimizers for~\cref{e:discrete-problem} has a subsequence that converges weakly in ${W^{1,p}(\Omega; \bR^{m})}$ to a  minimizer $\uu^*$ of~\cref{e:setup:min-problem}.
\end{proposition}
\begin{proof}
     Let $\uu^*$ be a minimizer for \cref{e:setup:min-problem}. A vector-valued version of \cref{th:bounded-density} gives a bounded sequence $\{\uu_h\}_{h>0}$ with $\uu_h\in \smash{\feSpaceU}$ for every $h$ that converges to $\uu^*$ strongly in $W^{1,p}(\Omega;\bR^m)$. We claim that this sequence satisfies $\calF(\uu_h)\to\calF(\uu^*)$. Indeed, since $\uu_h\to \uu^*$ strongly in $W^{1,p}(\Omega;\bR^m)$ we have $\uu_h \to \uu$ and $\nabla \uu_h \to \nabla \uu$ pointwise almost everywhere on $\Omega$. Further, by assumption~\ref{ass:growth} and the Sobolev embeddings~\cite[\S10.9]{Alt2016},
    \begin{equation*}
    \int_{\Omega} \vert f(x,\uu_h,\nabla \uu_h) \vert \, \dVolume
    \leq \int_{\Omega} B(1 + \vert \uu_h \vert^q + \vert \nabla \uu_h \vert^p) \, \dVolume
    \leq B\vert \Omega \vert + C \|\uu_h\|_{W^{1,p}}
    \end{equation*}
    for some constant $C$. The sequence $\{f(\cdot,\uu_h(\cdot),\nabla \uu_h(\cdot))\}_{h>0}$ is thus bounded in $L^1(\Omega)$ and converges pointwise almost everywhere to $f(\cdot,\uu(\cdot),\nabla \uu(\cdot))$. It must therefore converge also in $L^1(\Omega)$ \cite[Lemma~8.3]{Robinson2001}, whence $\calF(\uu_h) \to \calF(\uu^*)$.

    Next, let $\{u_h\}_{h>0}$ be the sequence converging to $u^*$ constructed above and let $u_h^*$ for $h>0$ be minimizers for \cref{e:discrete-problem}. Assumption \ref{ass:growth} ensures the sequence $\{u_h^*\}_{h>0}$ is bounded in $\smash{W^{1,p}_0(\Omega;\bR^m)}$, so it has a weakly convergent subsequence. Pass to this subsequence without relabelling and let $u_0 \in W^{1,p}_0(\Omega;\bR^m)$ be the weak limit. Since assumptions~\ref{ass:growth}--\ref{ass:quasiconvexity} guarantee that the functional $\calF$ is weakly lower semicontinuous \cite[Theorem~1.13]{Dacorogna2008}, and since $\calF(\uu_h) \geq \calF(\uu_h^*)$ by definition of $\uu_h^*$,
    \begin{align*}
        \calF^* = \calF(u^*) = \lim_{h\to0} \calF(u_h) \geq \liminf_{h\to 0} \calF(u_h^*) \geq \calF(u_0) \geq \calF^*.
    \end{align*}
    All inequalities must in fact be equalities, so we conclude that $u_0$ is a minimizer for \cref{e:setup:min-problem} and that $\calF_h^* = \calF(u_h^*) \to \calF(u_0)=\calF^*$.
\end{proof}

Finally, we remark that in \cref{e:unconstrained-objective-discretization} we have assumed the integration with respect to the $x$ variable can be performed exactly. However, our convergence result can be extended to account for numerical quadrature errors at least  under additional convexity assumptions on the function $z\mapsto f(x,y,z)$ \cite{Ortner2004}.
\section{Sparse moment-SOS relaxation}
\label{s:sos}

Set $\beta_h = \beta(h^{-1})$. For each mesh size $h$, the POP \cref{e:pop} requires minimizing the (generally nonconvex) $N$-variate polynomial 
\begin{equation}\label{e:phi-poly}
    \Phi(\xi) = \sum_{e=1}^{\Nel} f_e(\bF_e\xi)
\end{equation}
over the box $[-\beta_h,\beta_h]^N$. This is an NP-hard problem in general~\cite{Murty1987}. However, one can use \emph{moment-SOS relaxations} \cite{Lasserre2001,Laurent2009,Lasserre2015,Parrilo2013} to relax it into a hierarchy of convex SDPs, whose solutions provide convergent lower bounds on $\calF_h^*$ and, in some cases, approximations to the corresponding minimizers. For small $h$, these SDPs can be solved in practice only if their formulation exploits the special structure of $\Phi(\xi)$.
Various strategies are available for this \cite{Waki2006,Lasserre2006,Wang2021tssos,Wang2021chordal-tssos} (see also \cite{Zheng2021arc} for a general overview). Here, we use the so-called \emph{correlative sparsity} strategy from \cite{Waki2006,Lasserre2006}, which exploits the sparse couplings between $\xi_1\ldots\xi_N$ implied by the additive structure of $\Phi$. Note that this is a generic property of POPs obtained through an FE discretization of a variational problem, and that each $f_e$ depends in general on every monomial in the local DOFs $\bF_e\xi$. Further structure of the polynomials $f_e$ can be exploited using recent refinements of the correlative sparsity strategy \cite{Wang2019,Wang2020,Wang2021tssos,Wang2021chordal-tssos,Zheng2021arc}, but this structure is problem-dependent. Here, we consider only the general case.

\subsection{Variable cliques}\label{ss:cliques}
We begin by introducing the notion of \emph{variable cliques}, which offer a flexible way to describe subsets of the variables $\xi_1,\ldots,\xi_N$ in~\cref{e:pop} that are considered to be coupled to each other.
\begin{definition}\label{def:cliques}
A family $\bP_1\xi, \ldots, \bP_K\xi$ of subsets of $\xi$ are a family of \emph{variable cliques} for the POP~\cref{e:pop} if:
\begin{enumerate}[\textup{(}a\textup{)}, leftmargin=*, widest=9, topsep=0.5ex]
    \item\label{p:cliques-1} $\bP_1 \xi \cup \cdots \cup \bP_K \xi = \xi$;
    \item\label{p:cliques-3} 
    For every $e \in \{1,\ldots,\Nel\}$, there exists $k \in \{1,\ldots,K\}$ such that $\bF_e \xi \subseteq \bP_k \xi$.
\end{enumerate}
\end{definition}

One usually requires also that $\bP_k \xi \not\subset \bP_\ell \xi$ if $k\neq \ell$, but this is not essential.
We are particularly interested in variable cliques that satisfy the so-called \emph{running intersection property} (RIP).

\begin{definition}\label{def:rip}
    A family of variable cliques $\bP_1\xi ,\ldots,\bP_K\xi$ satisfy the \emph{running intersection property} if, possibly after reordering, for every $k \in \{1,\ldots,K\}$ there exists $\ell \in \{1,\ldots,k-1\}$ such that $\bP_k\xi \cap \left(\bP_1\xi \cup \cdots \cup \bP_{k-1} \xi \right) \subseteq \bP_\ell \xi.$
\end{definition}

The local DOF sets $\bF_1\xi,\ldots,\bF_{\Nel}\xi$ are the finest family of variable cliques, meaning that each clique is as small as it can be. If the minimization problem \cref{e:setup:min-problem} is over a one-dimensional domain, these cliques satisfy the RIP because the connectivity of mesh elements implies $\smash{\bF_k\xi \cap \left(\bF_1\xi \cup \cdots \cup \bF_{k-1} \xi \right) \subseteq \bF_{k-1} \xi}$ for every $k\in\{1,\ldots,\Nel\}$. However, the RIP fails if the Dirichlet boundary conditions in \cref{e:setup:min-problem} are replaced by periodic ones because periodicity implies $\bF_{\Nel}\xi \cap \left(\bF_1\xi \cup \cdots \cup \bF_{\Nel-1} \xi \right) \subseteq \bF_{\Nel-1} \cup \bF_{1} \xi$, which is not a subset of any other clique. The local DOF sets $\bF_1\xi,\ldots,\bF_{\Nel}\xi$ typically also fail to satisfy the RIP for integral minimization problems over multidimensional domains, even with Dirichlet boundary conditions.

\begin{figure}
    \centering
    \includegraphics[width=2.25cm]{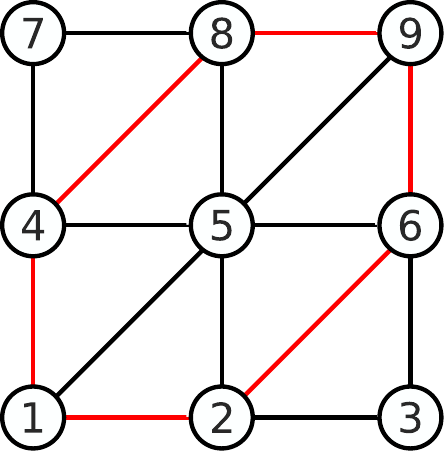}
    \hspace{50pt}
    \includegraphics[width=2.25cm]{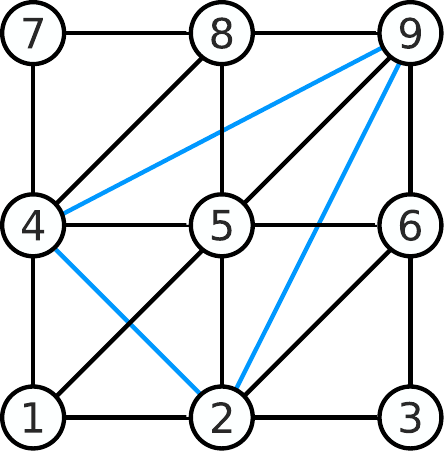}
    \caption{A non-chordal graph (left) and its chordal extension (right). The graph on the left is not chordal because the cycle $\{1,2,6,9,8,4\}$ (in red) does not have a `shortcut'. The chordal graph on the right is obtained by adding edges $(2,4)$, $(2,9)$ and $(4.9)$ (in blue) to the graph on the left.}
    \label{fig:cliques}
\end{figure}

On the other hand, one can always start from the local DOF sets $\bF_1\xi,\ldots,\bF_{\Nel}\xi$ to construct a different family of variable cliques, say $\bP_1\xi ,\ldots,\bP_K\xi$, that satisfy the RIP. The key observation \cite[Corollary~1]{Blair1993} is that the variable cliques $\bP_1\xi ,\ldots,\bP_K\xi$  satisfy the RIP if and only if they are the \emph{maximal cliques} (maximal fully connected components) of a \emph{chordal} graph with the variables $\xi_1,\ldots,\xi_N$ as its vertices (a graph is chordal if every cycle of length at least $4$ has a `shortcut', cf. \cref{fig:cliques}).
One can therefore proceed according to the following three steps:
\begin{enumerate}[itemsep=0.25ex,leftmargin=*, labelwidth=1ex,topsep=0.25ex]
    \item Construct the \emph{correlative sparsity graph} for \cref{e:pop}, which has $\xi_1,\ldots,\xi_N$ as its vertices and an edge from $\xi_i$ to $\xi_j$ if and only if $\xi_i,\xi_j \in \bF_e \xi$ for some $e\in\{1,\ldots,\Nel\}$.
    \item If the correlative sparsity graph is not chordal, perform a \emph{chordal extension} by adding edges to it until the extended graph is chordal.
    \item Identify the maximal cliques $\bP_1\xi ,\ldots,\bP_K\xi$ of this extended chordal graph. This can be done using algorithms with $O(N)$ complexity \cite{tarjan1984,berry2004}.
\end{enumerate}
The chordal extension step should add the smallest numbers of edges to minimize the size of the cliques $\bP_1\xi ,\ldots,\bP_K\xi$. Finding minimal chordal extensions is unfortunately an NP-complete problem~\cite{yannakakis1981}, but several heuristics are available \cite{Vandenberghe2015,Waki2006}; one example is to perform a symbolic Cholesky factorization of an approximately-minimum-degree permutation of the correlative sparsity graph's adjacency matrix.

\subsection{Riesz functional, moment and localizing matrices}
\label{ss:moment-matrices}
Beyond variable cliques, to build the moment-SOS relaxations of the POP~\cref{e:pop} we need to define the so-called \emph{Riesz functional}, \emph{moment matrices}, and \emph{localizing matrices}.

For a multi-index $\alpha \in \bN^N$ (with $\bN$ including $0$), let $\xi^\alpha = \xi_1^{\alpha_1}\cdots \xi_N^{\alpha_N}$ be the multivariate monomial with exponent $\alpha$ and let $|\alpha|=\alpha_1+\cdots+\alpha_N$ be its degree. For any integer $d\geq 0$, let $\bN^N_{d}$ be the set of exponents $\alpha \in \bN^N$ with degree $|\alpha|\leq d$, which contains $r(d)=\binom{N+d}{d}$ elements.
Given a vector $y \in \bR^{r(d)}$, the \emph{Riesz functional} $\riesz{y}$ acts on polynomials of degree up to $d$ via
\begin{equation*}
    \riesz{y}(p) := \sum_{\alpha \in \bN^N_{d}} p_\alpha y_\alpha 
    \qquad\text{where}\qquad p(\xi) = \sum_{\alpha \in \bN^N_{d}} p_\alpha \xi^\alpha\,.
\end{equation*}
Note that, for a fixed polynomial $p$, the map $y \mapsto \riesz{y}(p)$ is a linear function of $y$. 

Next, we define the moment and localizing matrices for the POP \cref{e:pop}. Fix a positive integer $\omega$ and a family of variable cliques $\bP_1\xi,\ldots,\bP_K\xi$. Let $N_k$ be the number of variables in clique $\bP_k \xi$ and let $\left[ \bP_k \xi \right]_\omega$ be the column vector listing the $\binom{N_k+\omega}{\omega}$ monomials in the $N_k$ variables $\bP_k \xi$ of degree up to $\omega$. (The precise ordering of these monomials is irrelevant.) Given a vector $y \in \bR^{r(2\omega)}$, the \emph{moment matrix} for clique $\bP_k \xi$ is the $\binom{N_k+\omega}{\omega}\times \binom{N_k+\omega}{\omega}$ symmetric matrix
\begin{equation*}
    M_k(y) = \riesz{y}\!\left( \left[ \bP_k \xi \right]_\omega \left[ \bP_k \xi \right]_\omega^\top \right),
\end{equation*}
where the Riesz functional is applied element-wise. Similarly, for each clique $\bP_k \xi$ and every $j\in\{1,\ldots,N\}$ for which $\xi_j\in \bP_k\xi$, the \emph{localizing matrix} corresponding to the constraint $\beta_h^2 - \xi_j^2\geq 0$ in \cref{e:pop} is the $\binom{N_k+\omega-1}{\omega-1}\times \binom{N_k+\omega-1}{\omega-1}$ symmetric matrix
\begin{equation*}
    L_{kj}(y) = \riesz{y}\!\left( (\beta_h^2 - \xi_j^2) \left[ \bP_k \xi \right]_{\omega-1} \left[ \bP_k \xi \right]_{\omega-1}^\top \right).
\end{equation*}
Observe that, for each $k$, there are $N_k$ distinct localizing matrices. Observe also that the functions $y\mapsto M_k(y)$ and $y\mapsto L_{kj}(y)$ are linear.

\begin{remark}\label{rem:general-cnstr-sos}
    Localizing matrices can be defined also if the constraints $\beta_h^2 \geq \xi_j^2$ in \cref{e:pop} are replaced by the constraints $\bF_1\xi,\ldots,\bF_{\Nel}\xi \in K_h$ for a set $K_h \subset \bR^{ms}$ defined via polynomial inequalities (cf. \cref{rem:general-cnstr}, where we had $m=1$). Precisely, suppose 
    \begin{equation*}
        K_h := \{z\in\bR^{ms}:\;g_1^h(z)\geq 0,\ldots,g_{q}^h(z)\geq 0\}
    \end{equation*}
    and let $d_l := \left\lceil \frac12 \deg(g_l^h)\right\rceil$. For each clique $\bP_k \xi$, one can define the localizing matrix
    \begin{equation*}
        L_{ke}^l(y) = \riesz{y}\!\left( g_l^h(\bF_e\xi) \left[ \bP_k \xi \right]_{\omega-d_l} \left[ \bP_k \xi \right]_{\omega-d_l}^\top \right)
    \end{equation*}
    for every $l=1,\ldots,q$ and every $e\in\{1,\ldots,\Nel\}$ for which $\bF_e\xi \subseteq \bP_k\xi$. Note that there is always at least one such index by condition \ref{p:cliques-3} in \cref{def:cliques}. This generalization is useful to reduce the number of localizing matrices (hence, the computational cost of the SDPs described in \cref{ss:sdp} below) when the cliques $\bP_1\xi,\ldots,\bP_K\xi$ contain more than a handful of elements. All convergence results in the following sections carry over to this general setting if the polynomials $g^h_1,\ldots,g^h_q$ satisfy the so-called \textit{Archimedean condition} \cite[\S2.4, \S6.1]{Lasserre2015}: there exist $r\in \bR$ and polynomials $\sigma_0,\ldots,\sigma_q$ that are sums of squares of other polynomials such that $r-|z|^2 = \sigma_0(z)+ g^h_1(z)\sigma_1(z)+\cdots+g^h_q(z)\sigma_q(z)$.
    This is true in particular when $q=1$ and $g^h_q(z)=\beta_h^2 - |z|^2$.
\end{remark}

\subsection{The SDP relaxation} \label{ss:sdp}
We are now ready to introduce the hierarchy of semidefinite program (SDP) relaxations for POP \cref{e:pop}. Let $\bP_1\xi,\ldots,\bP_K\xi$ be a family of variable cliques according to \cref{def:cliques}. Let $d_\Phi$ be the degree of the polynomial $\Phi(\xi)$ in \cref{e:phi-poly}, which is to be minimized over the box $[-\beta_h,\beta_h]^N$. For every positive integer $\omega$ such that $2\omega \geq \max\{2,d_\Phi\}$, called the \emph{relaxation order}, the degree-$2\omega$ moment-SOS relaxation of \cref{e:pop} is the SDP
\begin{align}\label{e:sdp-relaxation}
    \lambda^*_{h,\omega}
    := \min_{y \in \bR^{r(2\omega)}} \quad &\riesz{y}(\Phi)\\ \nonumber
    \text{subject to} \quad
    &M_{k}(y) \succeq 0 \quad \forall k \in \{1,\ldots,K\},\\\nonumber
    &L_{kj}(y) \succeq 0 \quad \forall k \in \{1,\ldots,K\}, 
    \; \forall j: \xi_j \in \bP_k \xi,\\\nonumber
    &\riesz{y}(1) = 1.
\end{align}
This problem does not actually depend on all $r(2\omega)=\binom{N+2\omega}{2\omega}$ entries of the vector $y$. Instead, by the additive structure of $\Phi(\xi)$ and the clique-based structure of the moment and localizing matrices, it only depends on entries $y_\alpha$ with $\alpha$ such that $\alpha_i\alpha_j>0$ implies $\xi_i,\xi_j\in\bP_k\xi$ for some clique $\bP_k\xi$. The exact number of variables actually appearing in \cref{e:sdp-relaxation} depends on how the variable cliques overlap, but satisfies
\begin{equation*}
    \text{\# variables} \leq \sum_{k=1}^K \binom{N_k+2\omega}{2\omega},
\end{equation*}
where $N_k$ is the number of variables in clique $\bP_k\xi$.
This upper bound is much smaller than $r(2\omega)$ if the variable cliques are small ($N_k$ much smaller than the total number $N$ of DOF) and it is sharp if and only if the cliques $\bP_1\xi,\ldots,\bP_K\xi$ are pairwise disjoint.

It is well known \cite{Waki2006,Lasserre2006} that $\lambda^*_{h,\omega}$ bounds the optimal value $\calF_h^*$ of the POP \cref{e:pop} from below. To see why, note first that \cref{e:pop} is equivalent to minimizing the integral
\begin{equation*}
    \int_{\bR^N} \Phi(\xi) \,\mathrm{d} \mu_h(\xi)
\end{equation*}
over probability measures $\mu_h$ supported on the feasible set $[-\beta_h,\beta_h]^N$. Indeed, every such probability measure satisfies $\int_{\bR^N} \Phi(\xi) \,\mathrm{d} \mu_h(\xi)\geq \min_{\xi \in [-\beta_h,\beta_h]^N} \Phi(\xi)=\calF_h^*$ and this inequality is sharp for Dirac measures supported on minimizers of $\Phi$. Then, let $\mu_h^*$ be optimal and let $y^*$ be the vector of its moments, i.e.,
\begin{equation*}
    y^*_\alpha := \int_{\bR^N} \xi^\alpha\,\mathrm{d}\mu_h^*(\xi),\qquad\alpha\in\bN^N_{2\omega}.
\end{equation*}
The Riesz functional corresponding to $y^*$ satisfies
\begin{equation*}
    \riesz{y^*}(p) = \sum_{\alpha \in \bN^N_{2\omega}} p_\alpha \int_{\bR^N} \xi^\alpha\,\mathrm{d}\mu_h^*(\xi) = \int_{\bR^N} p(\xi)\,\mathrm{d}\mu_h^*(\xi)
\end{equation*}
for every polynomial $p$, so in particular $\riesz{y^*}(1)=1$. We also find
\begin{gather*}
    M_k(y^*) = \int_{\bR^N} \left[ \bP_k \xi \right]_\omega \left[ \bP_k \xi \right]_\omega^\top\,\mathrm{d}\mu_h^*(\xi) \succeq 0
    \\
    L_{kj}(y^*) = \int_{\bR^N} (\beta_h^2 - \xi_j^2)\,\left[ \bP_k \xi \right]_{\omega-1} \left[ \bP_k \xi \right]_{\omega-1}^\top\,\mathrm{d}\mu_h^*(\xi) \succeq 0
\end{gather*}
for every choice of $k$ and $j$, because the matrices being integrated are positive semidefinite on the support of $\mu_h^*$ by construction. Thus, $y^*$ is feasible for the SDP \cref{e:sdp-relaxation} and gives a cost value of $\riesz{y^*}(\Phi)=\int_{\bR^N} \Phi(\xi) \,\mathrm{d} \mu_h^*(\xi)=\calF_h^*$, which cannot be smaller than $\lambda^*_{h,\omega}$ for any relaxation order $\omega$.

The crucial observation for our purposes is that the lower bound $\lambda^*_{h,\omega}$ is asymptotically sharp as $\omega\to\infty$ when the variable cliques $\bP_1\xi,\ldots,\bP_K\xi$ used to construct the SDP \cref{e:sdp-relaxation} satisfy the running intersection property (cf. \cref{def:rip}). Precisely, we have the following particular version of a more general result on the convergence of sparse moment-SOS relaxation from \cite{Lasserre2006} (see also \cite{Grimm2007}).

\begin{theorem}[{Adapted from \cite[Theorem 3.6]{Lasserre2006}}]\label{th:sos-convergence}
    Suppose the variable cliques $\bP_1\xi ,\ldots,\bP_K\xi$ used to construct the SDP \cref{e:sdp-relaxation} satisfy the running intersection property. Let $\{y^\omega\}_{\omega}$ be a sequence of optimal solutions of \cref{e:sdp-relaxation} for increasing $\omega$.
    \begin{enumerate}[\textup{(}1\textup{)}, leftmargin=*, labelwidth=1ex, topsep=0.25ex]
        \item $\lambda^*_{h,\omega} \to \calF_h^*$ from below as $\omega \to \infty$.
        \item There exists a subsequence $\{y^\omega\}_{\omega}$ and a probability measure $\mu_h$ supported on the set of minimizers of the POP \cref{e:pop} such that, for every $\alpha \in \bN^N$,
        \begin{equation*}
            \lim_{\omega \to \infty} y^\omega_\alpha = \int_{\bR^N} \xi^\alpha\,\mathrm{d}\mu_h(\xi).
        \end{equation*}
    \end{enumerate}
\end{theorem}

The second part of this result implies in particular that if the POP \cref{e:pop} has a unique minimizer $\xi^*$ then $y^\omega_\alpha \to (\xi^*)^\alpha$ as $\omega$ is raised, so one obtains increasingly accurate approximations to $\xi^*$ from the optimal solution of the SDP \cref{e:sdp-relaxation}. If \cref{e:pop} admits multiple minimizers, one can recover a subset of these from an optimal SDP solution $y^\omega$ at a given relaxation order $\omega$ if the moment matrices $M_k(y^\omega)$ satisfy a technical rank degeneracy condition called the \emph{flat extension condition}. We refer interested readers to \cite[\S3.3]{Lasserre2006} and references therein for the details. We remark only that the flat extension condition often holds in practice for large enough $\omega$, but it is not currently known if it holds generically unless \cref{e:sdp-relaxation} is constructed using the single `dense' clique $\bP_1\xi=\xi$ \cite{Nie2014}. At present, however, `dense' SDP relaxations are computationally intractable for POPs with more than a few tens of variables. 

\begin{remark}
    The SDP in \cref{e:sdp-relaxation} is dual to a maximization problem over \emph{sum-of-squares} (SOS) polynomials of degree $2\omega$. 
    Precisely, let $\Sigma_{N,\omega}$ be the set of $N$-variate SOS polynomials of degree $2\omega$ and note that $s \in \Sigma_{N,\omega}$ if and only if $s(\xi)=\smash{[\xi]_\omega^\top Q [\xi]_\omega}$ for some positive semidefinite matrix $Q$~\cite{Lasserre2001,Parrilo2003,Laurent2009,Parrilo2013}. One can show that~\cite{Waki2006,Lasserre2006}
    \begin{align}\label{e:wsos}
        \lambda^*_{h,\omega}=\max\bigg\{ \lambda\in\bR : \;
        \Phi(\xi)-\lambda=
        \sum_{k=1}^K
        \Big[s_{k}(\bP_k\xi)+\!\!\sum_{j:\,\xi_j\in\bP_k\xi} \!\!(\beta_h^2-\xi_j^2)r_{kj}(\bP_k\xi) \Big],&
        \\[-1ex] \nonumber
        s_k \in \Sigma_{N_k,\omega},
        \,\, r_{kj} \in \Sigma_{N_k, \omega-1}&\;
        \bigg\}.
    \end{align}
    The intuition behind this maximization problem is that $\lambda$ is a lower bound for $\Phi$ on a compact set if and only if $\Phi(\xi)-\lambda \geq 0$ on that set. The constraints in~\cref{e:wsos} provide a \emph{weighted} SOS certificate of nonnegativity for $\Phi(\xi)-\lambda$ on the compact box  $[-\beta_h,\beta_h]^N$, because SOS polynomials are globally nonnegative and $\beta_h^2 - \xi_j^2 \geq 0$ for every component $\xi_j$ of $\xi \in [-\beta_h,\beta_h]^N$. \Cref{th:sos-convergence} follows from a representation theorem from semialgebraic geometry, known as \emph{Putinar's Positivstellensatz} \cite{Putinar1993,Lasserre2001}, and a further refinement for polynomials with sparsely coupled variables \cite{Lasserre2006,Grimm2007}. These theorems guarantee the existence of weighted SOS representations like that in \cref{e:wsos} for positive polynomials on semialgebraic sets satisfying the Archimedean condition mentioned in \cref{rem:general-cnstr-sos}. Although the SOS maximization problem in~\cref{e:wsos} has a more intuitive meaning than its dual SDP in \cref{e:sdp-relaxation}, it does not lend itself to recovering optimal $\xi$ and is thus less convenient for our purposes.
\end{remark}

\subsection{Comments on computational complexity}\label{ss:complexity}
We conclude this section with a brief discussion of the computational complexity of the SDP \cref{e:sdp-relaxation}. For simplicity, let us assume that the $K$ variable cliques $\bP_1\xi,\ldots,\bP_K\xi$ have the same size, denoted by $N_c$. Let us also assume for definiteness that the SDP \cref{e:sdp-relaxation} is solved using a second-order interior-point algorithm (see, e.g., \cite{Ye1997,Nesterov2003,Nemirovski2006,Monteiro2000}), which are presently the state-of-the-art methods to find optimal solutions with high accuracy. Then, as explained in \cite[\S1.2]{Papp2019} (see also \cite{Monteiro2000}), for a fixed relaxation order $\omega$  each iteration of the algorithm has a running time of $\smash{O( \kappa^{1.5} \rho^{6.5})}$, where $\kappa = K(N_c + 1)$ is the number of linear matrix inequalities in the SDP \cref{e:sdp-relaxation} and $\rho=\smash{\binom{N_c + \omega}{\omega}}$ is their size.
This is in fact a conservative estimate, as we have ignored for simplicity that not all matrices in \cref{e:sdp-relaxation} have size $\rho \times \rho$ and that a careful implementation of sparse linear algebra operations often affords significant savings. Nevertheless, even more precise complexity estimates increase very quickly with the clique size $N_c$. It is thus usually easier to solve SDP relaxations of POPs with a large number of small variable cliques.

In our particular setting, where \cref{e:sdp-relaxation} is the relaxation of a POP obtained through an FE discretization, the smallest possible variable cliques are the local DOF sets $\bF_1\xi,\ldots,\bF_{\Nel}\xi$ (cf. \cref{ss:cliques}), whose size is determined by the element type. It seems therefore convenient to choose finite elements with the smallest possible number of local DOFs. In particular, while it is possible to use high-order order elements, we expect that the increase in precision and consequent potential to use a coarser mesh will not offset the large increase in computational complexity caused by the larger number of DOFs per element, which leads to larger variable cliques. It would be interesting to see if `static condensation' techniques from classical implementations of FE schemes could be extended to SDP solvers in order to more efficiently handle DOFs that belong to a single element, thereby alleviating the increase in computational complexity for high-order elements. This possibility remains unexplored.
\section{Convergence of the overall numerical strategy}
\label{s:convergence}
In \cref{s:fe,s:sos}, we discussed the two steps of our `discretize then relax' approach separately. First, we showed how to discretize the variational problem \cref{e:setup:min-problem} into a convergent hierarchy of POPs indexed by the mesh size $h$. Then, we showed how to relax each POP into a convergent hierarchy of SDPs indexed by a relaxation order $\omega$. We now put the two together and show that if \cref{e:setup:min-problem} has a unique global minimizer, then it is approximated arbitrarily well weakly in $\smash{W^{1,p}(\Omega;\bR^m)}$ by the functions
\begin{equation}\label{e:approx-minimizer}
    \uu_{h,\omega}^*(x) := \sum_{j=1}^N y^\omega_{e_j} \varphi_j(x),
\end{equation}
where $y^\omega$ is the optimal solution of SDP \cref{e:sdp-relaxation} and $e_j$ is the $N$-dimensional unit vector pointing in the $j$-th coordinate direction. This establishes the overall convergence of our numerical strategy, which is the main theoretical result of the paper.

\begin{theorem}\label{th:main-result}
    Suppose that:
    \begin{enumerate}[\textup{(}a\textup{)},leftmargin=*,widest=9,topsep=0.5ex]
        \item\label{ass:main:well-posed} The variational problem \cref{e:setup:min-problem} satisfies \cref{ass:polynomial-f,,ass:for-wlsc}. In particular, its minimum $\calF^*$ is attained by a minimizer $u^*$.
        \item\label{ass:main:fe-approx} The finite element sets $U_h^\beta$ used to derive the POP \cref{e:pop} satisfy the conditions of \cref{th:bounded-density}.
        \item\label{ass:main:rip} For each mesh size $h$, the variable cliques used to construct the SDP \cref{e:sdp-relaxation} satisfy the running intersection property (cf. \cref{def:rip}).
    \end{enumerate}
    Then, the following statements hold:
    \begin{enumerate}[\textup{(}i\textup{)},leftmargin=*,widest=9,topsep=0.5ex]
        \item\label{th:main:unc-bound-conv}
        The optimal value $\lambda_{h,\omega}^*$ of SDP \cref{e:sdp-relaxation} satisfies
        \begin{equation*}
            \lim_{h \to 0} \lim_{\omega \to \infty} \lambda^*_{h,\omega} = \calF^*.
        \end{equation*}
        \item\label{th:main:unc-minimizer-conv}
        If the minimizer $u^*$ is unique, then the function $u^*_{h,\omega}$ defined in \cref{e:approx-minimizer} converges to $u^*$ weakly in $\smash{W^{1,p}(\Omega;\bR^m)}$. Precisely, 
        \begin{equation*}
           \lim_{h \to 0} \lim_{\omega \to \infty} \left\vert \mathcal{L}(u^*_{h,\omega})- \mathcal{L}(u^*) \right\vert = 0
        \end{equation*}
        for every bounded linear functional $\mathcal{L}$ on $\smash{W^{1,p}(\Omega;\bR^m)}$.
        In particular, $u^*_{h,\omega} \to u^*$ strongly in $\smash{L^q(\Omega;\bR^m)}$ for every $q < p^*$, where $p^*$ is the Sobolev conjugate of $p$.
    \end{enumerate}  
\end{theorem}

\begin{proof}
    By assumption \ref{ass:main:rip}, for each mesh size $h$ we can use \cref{th:sos-convergence} to find $\lambda_{h,\omega}^*\to \calF^*_h$ as $\omega \to \infty$. Assumptions \ref{ass:main:well-posed} and \ref{ass:main:fe-approx}, instead, guarantee that $\calF_h^*\to\calF^*$ as $h \to 0$ by applying  \cref{th:unc-gamma-conv}. Combining these two limits yields statement~\ref{th:main:unc-bound-conv}.

    To establish statement~\ref{th:main:unc-minimizer-conv}, it suffices to prove that $u_{h,\omega}^*$ converges weakly to $u^*$ in $\smash{W^{1,p}(\Omega;\bR^m)}$ because the strong convergence in $\smash{L^q(\Omega;\bR^m)}$ for $q < p^*$ follows from the Rellich--Kondrachov theorem. Recall from \cref{th:sos-convergence} that for every mesh size $h$ there exists a probability measure $\mu_h$ supported on the set of minimizers of the discrete POP~\cref{e:pop} such that
    \begin{equation*}
        \lim_{\omega \to \infty} y^\omega_{e_j} = \int \xi_j \,{\rm d}\mu_h(\xi)
        \qquad \forall j\in \{1,\ldots,N\}.
    \end{equation*}
    As $\omega$ is raised, therefore, we find for every $x \in \overline{\Omega}$ that
    \begin{equation}
        \lim_{\omega \to \infty} u_{h,\omega}^*(x)
        = \lim_{\omega \to \infty} \sum_{j=1}^N y^*_{e_j} \varphi_j(x)
        = \int \bigg( \sum_{j=1}^N \xi_j \varphi_j(x) \bigg) {\rm d}\mu_h(\xi) =: u^*_{h,\infty}(x).
    \end{equation}
    Observe that $u^*_{h,\infty}$ is a convex combination of minimizers of the discrete problem~\cref{e:discrete-problem}.
    Since $\Omega$ has a compact closure, and since the functions $u_{h,\omega}^*$ and $u_{h,\infty}^*$ belong to a finite-dimensional FE space, we conclude that $u_{h,\omega}^* \to u^*_{h,\infty}$ uniformly in $\Omega$ and, in fact, in any norm, including $\|\cdot\|_{W^{1,p}}$. Then, for every bounded linear functional $\mathcal{L}$ on $\smash{W^{1,p}(\Omega;\bR^m)}$ we have 
    \begin{equation}\label{e:main-proof:limit-1}
        \lim_{\omega \to \infty} \left\vert \mathcal{L}(u^*_{h,\omega})- \mathcal{L}(u^*_{h,\infty}) \right\vert
        \leq \lim_{\omega \to \infty} \left\| \mathcal{L} \right\|
        \left\|  u^*_{h,\omega} - u^*_{h,\infty} \right\|_{W^{1,p}} = 0.
    \end{equation}

    Next, we prove that $u^*_{h,\infty}$ converges to $u^*$ weakly in $\smash{W^{1,p}(\Omega;\bR^m)}$ as $h\to 0$. Let $\mathcal{M}_h$ be the set of optimizers for the discrete variational problem~\cref{e:discrete-problem}, i.e.,
    \begin{equation*}
        \mathcal{M}_h := \left\{ u \in \feSpaceU :\; u(x)=\sum_{j=1}^N \xi^*_j \varphi_j(x),\; \xi^* \text{ is optimal for the POP \cref{e:pop}}\right\}.
    \end{equation*}
    Observe that we may identify the optimal probability measure $\mu_h$ above with a probability measure on $\feSpaceU$ supported on $\mathcal{M}_h$, also denoted by $\mu_h$ with a slight abuse of notation. Observe also that  $\mathcal{M}_h$ is a compact subset of $W^{1,p}(\Omega;\bR^m)$ because it is a closed subset of $\feSpaceU$, which is compact in $W^{1,p}(\Omega;\bR^m)$. Then, for every bounded (hence, continuous) linear functional $\mathcal{L}$ on $\smash{W^{1,p}(\Omega;\bR^m)}$ and every mesh size $h$ we can select 
    $$u_h^*\in\argmax_{u\in \mathcal{M}_h}|\mathcal{L}(u) - \mathcal{L}(u^*)|.$$
    Since $\{u_h^*\}_{h>0}$ is a sequence of optimizers for \cref{e:discrete-problem}, every subsequence has a further subsequence that converges weakly in $\smash{W^{1,p}(\Omega;\bR^m)}$ to the unique minimizer $u^*$ of~\cref{e:setup:min-problem} by \cref{th:unc-gamma-conv}. A straightforward contradiction argument then shows that the entire sequence $\{u_h^*\}_{h>0}$ must converge weakly to $u^*$ in $\smash{W^{1,p}(\Omega;\bR^m)}$. Then,
    \begin{equation*}
        \lim_{h\to 0} \; \max_{u \in \mathcal{M}_h} \left\vert  \mathcal{L}(u) - \mathcal{L}(u^*)\right\vert
        = \lim_{h\to 0} \left\vert  \mathcal{L}\left(u_h^*\right)  - \mathcal{L}(u^*)\right\vert
        = 0.
    \end{equation*}
    Combining this with the identity
    \begin{equation*}
        \mathcal{L}\left(u^*_{h,\infty}\right)
        = \mathcal{L}\left( \int_{\mathcal{M}_h} u \, {\rm d}\mu_h(u) \right)
        = \int_{\mathcal{M}_h} \mathcal{L}\left(  u  \right) \, {\rm d}\mu_h(u),
    \end{equation*}
    which follows from applying Jensen's inequality to the linear (hence, both convex and concave) function $u\mapsto \mathcal{L}(u)$, we obtain
    \begin{align}\label{e:main-proof:limit-2}
        \lim_{h\to 0} \left\vert 
        \mathcal{L}\left(u^*_{h,\infty}\right) - 
        \mathcal{L}\left(u^*\right)
        \right\vert
        &= 
        \lim_{h\to 0} \left\vert 
        \int_{\mathcal{M}_h} \mathcal{L}\left(  u  \right)
        - \mathcal{L}\left(u^*\right)\, {\rm d}\mu_h(u)
        \right\vert
        \\ \nonumber
        &\leq 
        \lim_{h\to 0}
        \max_{u \in \mathcal{M}_h} \left\vert\mathcal{L}\left(  u  \right) 
        - \mathcal{L}\left(u^*\right)
        \right\vert
        = 0.
    \end{align}
    
    Finally, to prove statement~\ref{th:main:unc-minimizer-conv} it suffices to estimate
    \begin{equation*}
        \vert \mathcal{L}(u^*_{h,\omega})- \mathcal{L}(u^*) \vert \leq \vert \mathcal{L}(u^*_{h,\omega})- \mathcal{L}(u^*_{h,\infty}) \vert + \vert \mathcal{L}(u^*_{h,\infty})- \mathcal{L}(u^*) \vert
    \end{equation*} 
    and take first $\omega\to\infty$ using \cref{e:main-proof:limit-1}, then $h\to 0$ using \cref{e:main-proof:limit-2}.
\end{proof}

It is important to observe that while the optimal values $\lambda_{h,\omega}^*$ of SDP \cref{e:sdp-relaxation} converge to the global minimum $\calF^*$ of the integral minimization problem \cref{e:setup:min-problem}, they are neither upper nor lower bounds on it. Instead, they are lower bounds for the upper bound $\calF^*_h$ on $\calF^*$. On the other hand, optimal solutions of the SDP \cref{e:sdp-relaxation} can be used to obtain an upper bound on $\calF^*$ for given finite values of $h$ and $\omega$, since it suffices to construct the function $u_{h,\omega}^*$ from \cref{e:approx-minimizer} and evaluate the functional $\calF(u_{h,\omega}^*)$. Unfortunately, however, the weak convergence of $u_{h,\omega}^*$ to $u^*$ in $\smash{W^{1,p}(\Omega;\bR^m)}$ guaranteed by \cref{th:main-result} does not imply that $\calF(u_{h,\omega}^*) \to \calF^*$ as one lets first $\omega\to\infty$ and then $h\to 0$. This is because assumptions~\ref{ass:growth}--\ref{ass:quasiconvexity} in \cref{s:setup} ensure that the functional $\calF(u)$ is weakly lower-semicontinuous, but not weakly continuous.

One simple case in which our upper bounds $\calF(u_{h,\omega}^*)$ do converge to $\calF^*$ is when the integrand $f$ in problem \cref{e:setup:min-problem} is the sum of a continuous function of $u$ and a convex function of the gradient $\nabla u$. Specifically, we can establish the following result.

\begin{theorem}\label{th:conv-sos-ub}
    In addition to the assumptions in \cref{th:main-result}, suppose the integrand $f$ in problem \cref{e:setup:min-problem} satisfies $f(x,y,z) = f_0(x,z) + f_1(x,y)$ where the function $z\mapsto f_0(x,z)$ is convex for almost every $x \in \Omega$. Let $u_{h,\omega}^*$ be defined as in \cref{e:approx-minimizer}. If problem \cref{e:setup:min-problem} has a unique global minimizer, then $\lim_{h \to 0} \lim_{\omega \to \infty} \calF( u_{h,\omega}^* ) = \calF^*$.
\end{theorem}

\begin{proof}
    Recall from the proof of \cref{th:main-result} that $u_{h,\omega}^*$ converges in any norm to a function $u_{h,\infty}^*$ that is a convex combination of optimizers for the discrete minimization problem~\cref{e:discrete-problem}. The first term on the right-hand side of the inequality
    \begin{equation*}
        \left\vert \calF( u_{h,\omega}^* ) - \calF^*\right\vert
        \leq 
        \left\vert \calF( u_{h,\omega}^* ) - \calF( u_{h,\infty}^* )\right\vert
        + \left\vert \calF( u_{h,\infty}^* ) - \calF^*\right\vert
    \end{equation*}
    vanishes as $\omega\to\infty$ because, by the same arguments in the proof of \cref{th:unc-gamma-conv}, our assumptions on $f$ make the functional $\calF(u)$ strongly continuous on $\smash{W^{1,p}_0(\Omega;\bR^m)}$.
    To conclude the proof, therefore, it suffices to show that $\smash{\calF ( u_{h,\infty}^* )} \to \calF^*$ as $h\to 0$. 
    
    For this, let us write $\calF(u) =\calF_0(u) + \calF_1(u)$ with
    $\calF_0(u) = \int_\Omega f_0(x,\nabla u)\, \dVolume$ and
    $\calF_1(u) = \int_\Omega f_1(x, u)\, \dVolume$.
    Our assumptions on $f_0$ and $f_1$ ensure that the functional $\calF_0$ is convex on $\smash{W^{1,p}(\Omega;\bR^m)}$, while $\calF_1$ is sequentially weakly continuous on the same space. The latter statement follows from arguments similar to those in the proof of \cref{th:unc-gamma-conv} because weak convergence in $\smash{W^{1,p}(\Omega;\bR^m)}$ implies strong convergence in $\smash{L^q(\Omega;\bR^m)}$ for all $q < p^*$ (the Sobolev conjugate of $p$) and $f_1(x,y)$ grows no faster than $|y|^q$ for $q<p^*$ by assumption~\ref{ass:growth}. 
    
    Next, we recall from the proof of \cref{th:main-result} that $u_{h,\infty}^* = \int_{\mathcal{M}_h} u \, {\rm d}\mu_h(u)$ for a probability measure $\mu_h$ supported on the set $\mathcal{M}_h$ of minimizers for the discrete problem \cref{e:discrete-problem}. In particular  $\calF^*_h=\calF(u)$ for any $u\in\mathcal{M}_h$. Then, writing $u^*$ for the (unique) global minimizer of problem \cref{e:setup:min-problem}, we can apply Jensen's inequality to the convex functional $\calF_0$ to estimate
    \begin{align*}
        0 & \leq \calF( u_{h,\infty}^* ) - \calF^*
        \\
        &= \calF_0( u_{h,\infty}^* ) - \calF^* + \calF_1( u_{h,\infty}^* )
        \\
        &\leq 
        \int_{\mathcal{M}_h} \calF_0( u ) \,{\rm d}\mu_h(u)
        - \calF^* + \calF_1( u_{h,\infty}^* ) 
        \\
        &= 
        \int_{\mathcal{M}_h} \calF( u )  \,{\rm d}\mu_h(u)
        - \calF^* + \calF_1( u_{h,\infty}^* )
        -\int_{\mathcal{M}_h} \calF_1( u ) \,{\rm d}\mu_h(u)
        \\
        &=
        \calF^*_h - \calF^* 
        +\calF_1( u_{h,\infty}^* ) - \calF_1( u^* )
        +\int_{\mathcal{M}_h} \calF_1( u^* ) - \calF_1( u ) \,{\rm d}\mu_h(u)
        \\
        &\leq 
        \calF^*_h - \calF^* 
        +\calF_1( u_{h,\infty}^* ) - \calF_1( u^* )
        +\max_{u\in\mathcal{M}_h} \left\vert \calF_1( u^* ) - \calF_1( u ) \right\vert.
    \end{align*}
    
    We now claim that the right-hand side of this inequality tends to zero as $h\to 0$, which immediately implies the desired convergence of $\calF ( u_{h,\infty}^* )$ to $\calF^*$. To see that the claim is true, note that $\calF^*_h\to\calF^*$ by \cref{th:unc-gamma-conv}. Next, recall from the proof of \cref{th:main-result} that $u_{h,\infty}^* \rightharpoonup u^*$ in $\smash{W^{1,p}(\Omega;\bR^m)}$, so $\calF_1( u_{h,\infty}^* ) \to \calF_1( u^* )$ by the weak continuity of $\calF_1$. Further, for each $h>0$ let 
    $u_h^* \in \argmax_{u\in \mathcal{M}_h}|\calF_1(u^*) - \calF_1(u)|$.
    Arguing as in the proof of \cref{th:main-result} we find that $u^*_h$ converges to $u^*$ weakly in $\smash{W^{1,p}(\Omega;\bR^m)}$, whence $\max_{u\in\mathcal{M}_h} \left\vert \calF_1( u^* ) - \calF_1( u ) \right\vert \to 0$.
\end{proof}

In fact, the last proof reveals that the computable upper bounds $\calF(u_{h,\omega}^*)$ converge to $\calF^*$ as $\omega\to\infty$ and $h\to0$ whenever $\calF$ is the sum of a convex functional $\calF_0$ and a weakly continuous functional $\calF_1$. We believe it should be possible to extend \cref{th:conv-sos-ub} also to general integrands $f(x,y,z)$ that are convex in their third argument using more technical steps similar to those in \cite[\S3.2.6 and \S8.2.4]{Dacorogna2008}, which allow one to ``freeze'' the lower-order terms when trying to show that $\calF(u_{h,\infty}^*)$ tends to $\calF^*$ as $h$ is reduced. For problems with nonconvex dependence on the gradient, instead,  we currently do not know if $\calF(u_{h,\omega}^*)$ could fail to converge to $\calF^*$.

If convergence of the upper bounds $\calF(u_{h,\omega}^*)$ cannot be established, one way to assess if they are nearly sharp in practice is to compare them to lower bounds on $\calF^*$ computed with the polynomial optimization techniques discussed in \cite{Korda2018MomentPDEs,Chernyavsky2023}. These lower bounds, however, are at present guaranteed to approach $\calF^*$ only for certain variational problems related to optimal constants in functional inequalities \cite{Chernyavsky2023}, or when the integrand $f$ in~\cref{e:setup:min-problem} is jointly convex in $u$ and its derivatives~\cite{FantuzziTobasco2023,Henrion2023convexity}. This convexity assumption is more restrictive than those in \cref{th:conv-sos-ub}, as it implies that the minimization problem~\cref{e:setup:min-problem} is convex. For convex problems, of course, local minimizers are also global ones and one should compute them with traditional numerical methods, whose computational complexity scales much better compared to our `discretize-then-relax' approach.
\section{Computational experiments}
\label{s:examples}

We now report on a series of computational experiments on examples that satisfy \cref{ass:polynomial-f,ass:for-wlsc} (or higher-order extensions thereof), as well as the convexity and separability assumptions in \cref{th:conv-sos-ub}. \Cref{ass:uniqueness} on the uniqueness of global minimizers also appears to be satisfied in all cases, but we have no proof. These experiments confirm our theoretical analysis and, crucially, showcase the practical convergence properties of our discretize-then-relax strategy. In particular, our computations demonstrate that the RIP, required by \cref{th:main-result}, is apparently not essential to observe convergence in practice. On the other hand, while the bounds $\beta_h$ on the DOFs of functions in the FE spaces $\feSpaceU$ do not play a key role in the proof of \cref{th:main-result}, their choice can strongly affect the practical performance of our numerical scheme.

Open-source code to reproduce the results presented in the following subsections can be downloaded from \url{https://github.com/giofantuzzi/fe-sos}. We used the MATLAB toolboxes \texttt{aeroimperial-yalmip}\footnote{\url{https://github.com/aeroimperial-optimization/aeroimperial-yalmip}} and \texttt{aeroimperial-spotless}\footnote{\url{https://github.com/aeroimperial-optimization/aeroimperial-spotless}} to implement the sparse moment-SOS relaxations from \cref{s:sos}. The resulting SDPs were solved using \texttt{MOSEK}~\cite{mosek} on a PC with 362GB RAM and two Intel\textsuperscript{\tiny\textregistered} Xeon\textsuperscript{\tiny\textregistered} Silver 4108 CPUs.

\subsection{Singularly perturbed two-well problem}
\label{ss:two-well-problem}

As our first example we consider the minimization of a singularly perturbed two-well functional,
\begin{equation}\label{e:two-well-problem}
    \min_{u \in W^{1,2}_0} \int_\Omega \varepsilon^2 \abs{\nabla u}^2 + (u+1)^2(u-2)^2 \, \dVolume,
\end{equation}
where $\Omega \subset \bR^2$ is a Lipschitz domain and $\varepsilon$ is a small parameter. Stationary points of this functional satisfy $u=-1$ or $u=2$ except for boundary or interior layers of $O(\varepsilon)$ width~\cite{Owen1990}. The minimizer satisfies $u=-1$ except for boundary layers. Here, we verify this for $\varepsilon = 0.1$ when $\Omega$ is a square, a circle, or an ellipse.

We discretized~\cref{e:two-well-problem} using conforming piecewise-linear Lagrange elements on a triangular mesh with $\Nel$ triangles of diameter no larger than $h=\sqrt{2}/k$ for $k=10$, $20$, $30$, $40$ and $50$. The DOF vector $\xi$ lists the values of $u$ at the mesh nodes in the interior of $\Omega$. The smallest possible variable cliques $\bF_1\xi,\ldots,\bF_{\Nel}\xi$ are the sets of nodal values of $u$ at the corners of each mesh triangle not intersecting the boundary.

We first set up and solved the SDP \cref{e:sdp-relaxation} using the DOF bound $\beta_h = \sqrt{2}/h$, the variable cliques $\bF_1\xi,\ldots,\bF_{\Nel}\xi$, and the relaxation order $\omega=2$. This is the smallest possible value for this example, as the integrand in \cref{e:two-well-problem} is quartic in $u$. Although our chosen variable cliques do not satisfy the RIP, so \cref{th:main-result} cannot be used to guarantee the convergence of the approximate minimizers $u^*_{h,\omega}$, \cref{fig:two-well-results} shows that the minimizer of~\cref{e:two-well-problem} is approximated with excellent accuracy. 

\begin{table}[]
    \caption{Number of cliques, maximum clique size, average clique size, and time required to solve the SDP relaxations of order $\omega=2$ for the POP discretization of problem~\cref{e:two-well-problem} on the square $\Omega = [-0.5,0.5]^2$.}
    \label{tab:two-well-results}
    \centering
    \small
    \begin{tabular}{c cccc c cccc}
    \hline
    &\multicolumn{4}{c}{Without running intersection} && \multicolumn{4}{c}{With running intersection}\\
    \cline{2-5} \cline{7-10}
    $h$ &cliques & max sz & avg sz & time (s) && 
    cliques & max sz & avg sz & time (s)\\
    $\sqrt{2}/10$ & 128 & 3 & 3 & 0.47 && 72 & 10 & 7.7 & 21.9 \\
    $\sqrt{2}/20$ & 648 & 3 & 3 & 2.65 && 342 & 20 & 14.3 & 15\,545 \\
    $\sqrt{2}/30$ & 1568 & 3 & 3 & 7.61 && 812 & 30 & 21 & N/A \\
    $\sqrt{2}/40$ & 2888 & 3 & 3 & 16.3 && 1482 & 40 & 27.7 & N/A \\
    $\sqrt{2}/50$ & 4608 & 3 & 3 & 29.3 && 2352 & 50 & 34.3 & N/A \\
    \hline
    \end{tabular}
\end{table}

\begin{figure}
    \centering
    \includegraphics[trim={1cm 1.4cm 0.25cm 0}, clip, width=\linewidth]{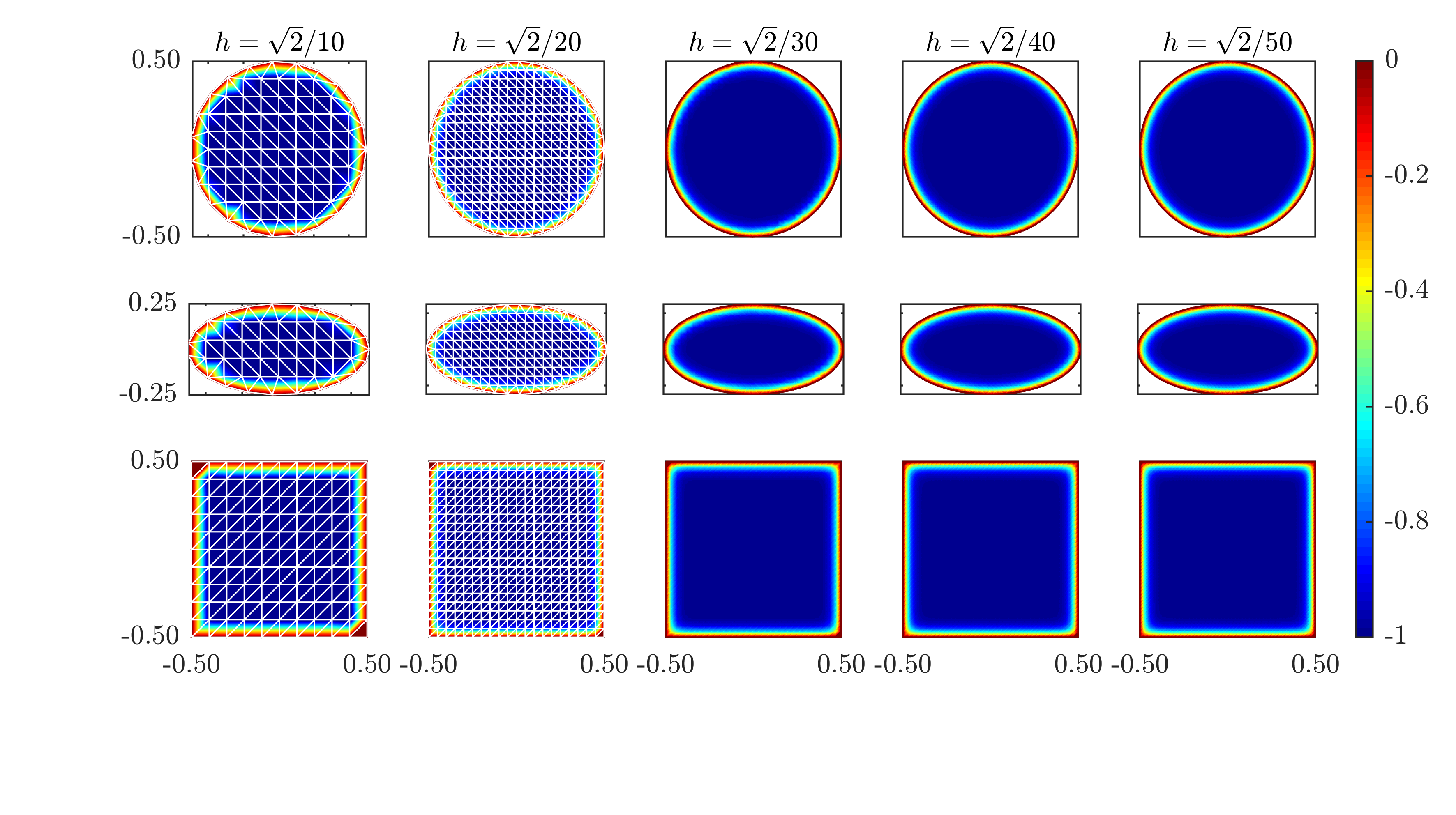}
    \vspace*{-16pt}
    \caption{Approximate minimizer for the two-well problem~\cref{e:two-well-problem} obtained for different mesh sizes $h$ with moment-SOS relaxation order $\omega = 2$ and DOF bound $\beta_h = \sqrt{2}/h$. Results for different mesh sizes $h$ are almost indistinguishable. The FE mesh is shown only for $h=\sqrt{2}/10$ and $\sqrt{2}/20$.}
    \label{fig:two-well-results}
\end{figure}

We then repeated the computation with a set of variable cliques that satisfy the RIP, which we constructed using the symbolic Cholesky factorization approach proposed in \cite{Waki2006} and outlined in \cref{ss:cliques}. The number of cliques, the maximal clique size, and the average clique size are listed in \cref{tab:two-well-results}. In this case, \cref{th:main-result} applies but we did not see any noticeable improvement on the quality of approximate minimizers. On the other hand, for the reasons given in \cref{{ss:complexity}}, the computational cost increased dramatically: as reported in \cref{tab:two-well-results}, the SDP \cref{e:sdp-relaxation} was intractable with our computational resources for all but the two coarsest meshes. For this example, therefore, enforcing the RIP not only seems unnecessary, but is also computationally prohibitive. Yet, it remains a necessary ingredient in our proof of \cref{th:main-result}.

\subsection{Swift--Hohenberg energy potential in 1D}
\label{ss:1d-sh-conforming}

Next, we minimize the Swift--Hohenberg energy potential on the one-dimensional interval $\Omega=(-\ell,\ell)$,
\begin{equation}\label{e:sh-1d-potential}
\calF^* := \inf_{u \in W^{2,2}_0(-\ell,\ell)}  \int_{-\ell}^{\ell}  \left(\partial_x^2 u + u \right)^2 - r u^2 - b u^3 + \tfrac12 u^4 \,\dVolume.
\end{equation}
Here, $W^{2,2}_0(-\ell,\ell)$ is the space of functions with two weak derivatives in $L^2(-\ell,\ell)$ and such that $u(\pm \ell)=\partial_x u(\pm \ell)=0$. The domain half-size $\ell$ and the constants $r$ and $b$ are problem parameters; here, we fix $\ell=32$, $r=0.3$ and $b=1.2$. As already remarked at the end of \cref{s:setup}, our numerical strategy and its convergence guarantees from \cref{th:main-result} extend to this problem despite the presence of second-order derivatives.

\begin{table}[]
	\caption{Lower bounds $\smash{\lambda^*_{h,\omega}}$ on the minimum of FE discretizations of problem~\cref{e:sh-1d-potential}, as a function of the relaxation order $\omega$ and the mesh size $h$. Reported wall times to compute each bound are in seconds.}
	\label{t:sh-1d-conforming-pop}
	\centering
	\small
	\begin{tabular}{c c cc c cc c cc}
		\hline
		&& \multicolumn{2}{c}{$\omega=2$} && \multicolumn{2}{c}{$\omega=3$} && \multicolumn{2}{c}{$\omega=4$} \\
		\cline{3-4} \cline{6-7} \cline{9-10}
		$h$ && $\lambda^*_{h,\omega}$ & Time (s) && $\lambda^*_{h,\omega}$ & Time (s) && $\lambda^*_{h,\omega}$ & Time (s)\\\hline
		$4$ &&  $-4.9227$ & $0.15$ &&  $-4.9049$ & $0.62$ &&  $-4.9049$ & $5.15$\\
		$2$ && $-25.1396$ & $0.20$ && $-25.1396$ & $1.29$ && $-25.1396$ & $9.92$\\
		$1$ && $-35.2759$ & $0.60$ && $-35.2360$ & $2.95$ && $-35.2360$ & $20.1$\\
		$1/2$ && $-37.6287$ & $1.23$ && $-35.3169$ & $7.76$ && $-35.3172$ & $50.4$\\
		$1/4$ && $-41.7708$ & $2.34$ && $-37.0518$ & $15.4$ && $-36.3883$ & $105$\\
		$1/8$ && $-46.5460$ & $6.13$ && $-45.4016$ & $33.6$ && $-41.5053$ & $267$\\
		\hline
	\end{tabular}
\end{table}

\begin{figure}
    \centering
    \includegraphics[width=\linewidth]{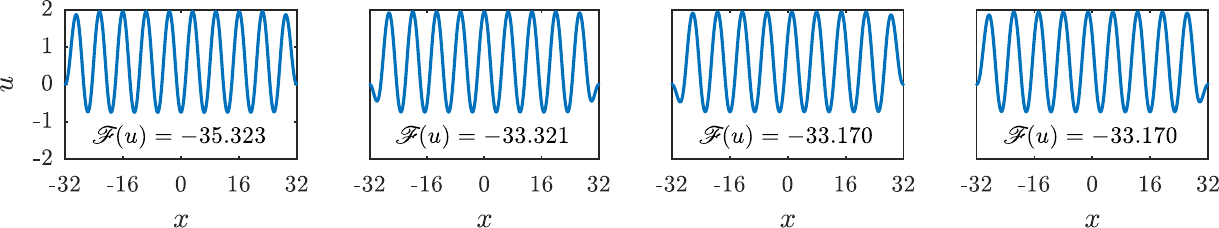}
    \vspace*{-16pt}
    \caption{Local minimizers for the Swift--Hohenberg problem~\cref{e:sh-1d-potential} obtained by solving~\cref{e:sh-1d-gradient-flow}. The corresponding values of the functional $\calF(u)$ being minimized in~\cref{e:sh-1d-potential} are also reported in each panel.}
    \label{fig:sh-1d-local-min}
\end{figure}

Local minimizers are stable steady solutions of the gradient-flow equation
\begin{equation}\label{e:sh-1d-gradient-flow}
\partial_t u = -\partial_x^4 u - 2 \partial_x^2 u - (1-r)u + \tfrac32 b u^2 - u^3, \qquad u(0,x)=u_0.
\end{equation}
Starting from 100 different initial conditions $u_0$, we found the four local minimizers plotted in \cref{fig:sh-1d-local-min}, which also reports the corresponding energy values.
To approximate global minimizers, instead, we apply our discretize-then-relax strategy using a $W^{2,2}_0$-conforming Hermite FE discretization on a uniform mesh $\mesh_h = \{(x_{e-1},x_{e})\}_{e=1}^{N}$ with $x_i = ih-\ell$ for $i=0,\ldots,N$, where $h=2\ell/N$ is the mesh size. On each element, the function $u$ is approximated by a cubic polynomial using the four nodal values $u(x_{e-1})$, $\partial_x u(x_{e-1})$, $u(x_e)$ and $\partial_x u(x_e)$ as the DOF. After enforcing the boundary conditions, the finest family of variable cliques consists of the $N-2$ DOF sets
\begin{equation*}
\bF_e \xi = \{ u(x_{e-1}), \partial_x u(x_{e-1}), u(x_e), \partial_x u(x_e)\} \quad \text{for} \quad e=2,\ldots,N-1.
\end{equation*}
Since these cliques satisfy the running intersection property, the convergence results in \cref{th:main-result} (suitably extended to problems with second-order derivatives) apply.

\begin{figure}
    \centering
    \includegraphics[width=\linewidth]{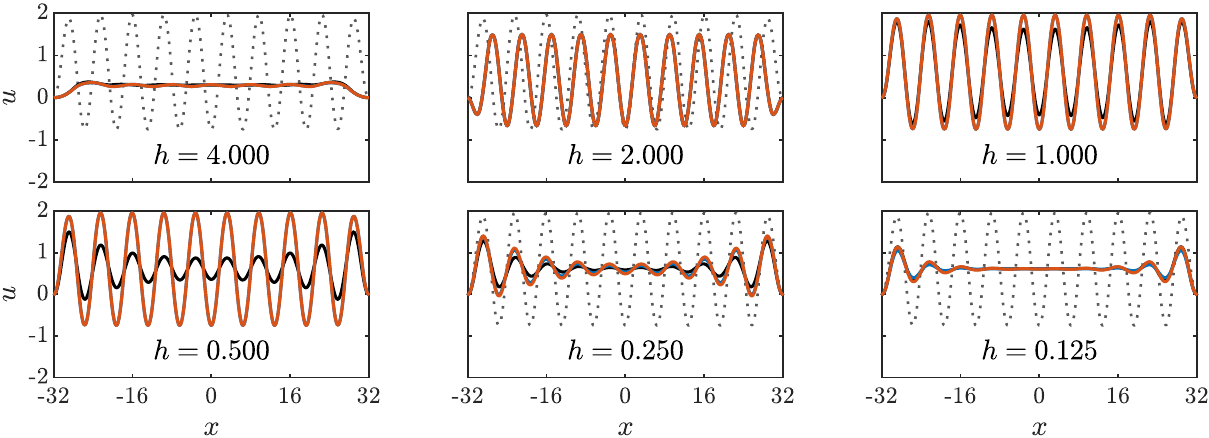}
    \vspace*{-16pt}
    \caption{Approximate minimizers for problem~\cref{e:sh-1d-potential} obtained with $\omega = 2$ (black), $3$ (blue) and $4$ (red), for different mesh sizes $h$. Results for $\omega=3$ and $4$ are almost indistinguishable. For each $h$, the DOF bound was $\beta_h=2/h$. Grey dotted lines show the best local minimizer found with~\cref{e:sh-1d-gradient-flow}.}
    \label{fig:sh-1d-conforming-sos-results}
    \vspace*{10pt}
    \centering
    \includegraphics[width=\linewidth]{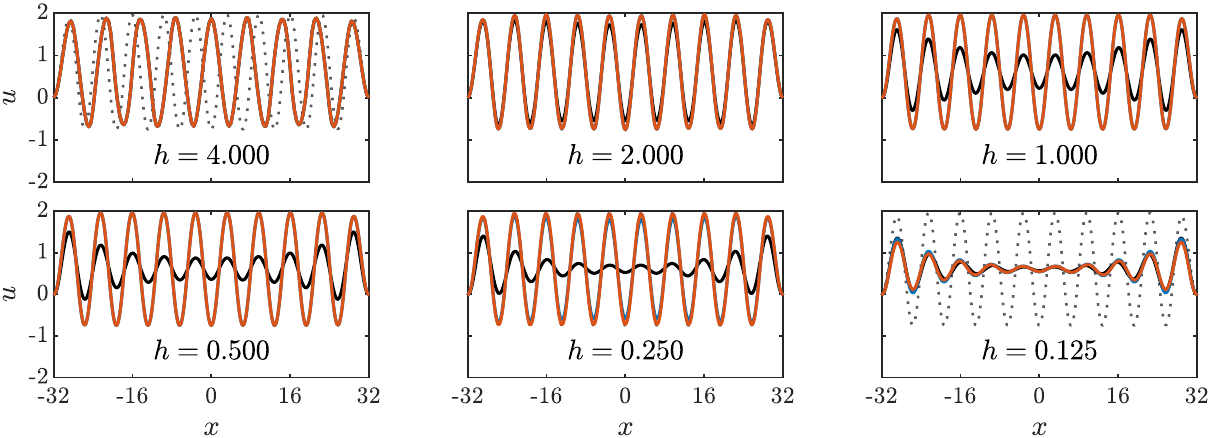}
    \vspace*{-16pt}
    \caption{Approximate minimizers for problem~\cref{e:sh-1d-potential} obtained with $\omega = 2$ (black), $3$ (blue) and $4$ (red), for different mesh sizes $h$. Results for $\omega=3$ and $4$ are almost indistinguishable. For each $h$, the DOF bound was $\beta_h=4$. Grey dotted lines show the best local minimizer found with~\cref{e:sh-1d-gradient-flow}.}
    \label{fig:sh-1d-conforming-sos-results-constand-bound}
\end{figure}

\Cref{fig:sh-1d-conforming-sos-results} shows the approximate global minimizers $u^*_{h,\omega}$ obtained with $\omega=2$, $3$ and $4$, DOF bound $\beta_h = 2/h$, and various mesh sizes $h$. The lower bounds $\lambda^*_{h,\omega}$ on the minimal discrete energy $\calF^*_h$ and time required to solve the SDP \cref{e:sdp-relaxation} are reported in \cref{t:sh-1d-conforming-pop} for each value of $\omega$ and $h$ . 
For the two coarsest meshes ($h=4$ and $2$) the bounds imposed on the degrees of freedom are too restrictive to obtain good results. For intermediate mesh sizes ($h=1$ and $\frac12$) the approximations obtained with $\omega=3$ and $4$ are virtually indistinguishable from the best local minimizer from \cref{fig:sh-1d-local-min}. This suggests the latter is indeed the global minimizer.
For the two finest meshes ($h=\frac14$ and $\frac18$), instead, even a relaxation order of $\omega=4$ is too small to obtain a good approximation to the global minimizer. This can be partly attributed to the increase in the imposed DOF bound: when this is held at the value of $\beta_h=4$ independently of $h$, we obtain excellent results with $\omega = 4$ for all but the finest mesh (see \cref{fig:sh-1d-conforming-sos-results-constand-bound}). Interestingly, if the functions $u^*_{h,\omega}$ plotted in \cref{fig:sh-1d-conforming-sos-results,fig:sh-1d-conforming-sos-results-constand-bound} are used to initialize the gradient flow equation \cref{e:sh-1d-gradient-flow}, then the ensuing solutions converge to the global minimizer of \cref{e:sh-1d-potential} for all but the coarsest mesh ($h=4$), for which the initial condition does not have the right oscillation frequency (nine peaks instead of ten, see the top-left panel in \cref{fig:sh-1d-conforming-sos-results-constand-bound}). At least for this example, therefore, our approach gives a good initial guess for traditional minimization methods even when it does not return an accurate approximation to the global minimizer.

Of course, \cref{th:main-result} guarantees that the approximations $u^*_{h,\omega}$ for a given mesh size $h$ will continue to improve for $\omega>4$. However, raising $\omega$ quickly increases the computational cost (cf. \cref{t:sh-1d-conforming-pop}). In practice, it is therefore essential to select good bounds on the DOF and a mesh size that is just fine enough to resolve the main features of the minimizer one seeks. Under-resolved approximate optimizers can then be refined using standard Newton iterations. Good DOF bounds can be chosen independently of the mesh size $h$ if \emph{a priori} uniform bounds on the minimizer of the integral minimization problem at hand can be derived through a separate analysis. In this case, one could even modify \cref{th:main-result} to remove the requirement that DOF bounds increase with $h$. On the other hand, we do not have a universal strategy to select appropriate values for $\omega$ and $h$, and we suspect that optimal choices are problem-dependent. Nevertheless, very little experimentation was required to obtain the results in this work, suggesting that our computational strategy does not require much tuning.
Moreover, for each mesh size $h$ one could in principle detect convergence in $\omega$ of the lower bounds $\lambda_{h,\omega}^*$ provided by the  SDP \cref{e:sdp-relaxation} by checking if the moment matrices satisfy the so-called flat extension condition; see \cite[\S3.3]{Lasserre2006} for details.

\subsection{Swift–Hohenberg energy potential in 2D}
\label{ss:2d-sh}

Finally, we consider a two-dimensional version of the previous example,
\begin{equation}\label{e:sh-2d-potential}
\inf_{u \in W^{2,2}_0(\Omega)}  \int_{\Omega}  \left( \Delta u + u \right)^2 - r u^2 - b u^3 + \tfrac12 u^4 \,\dVolume,
\end{equation}
on the rectangle $\Omega = [-\ell_x,\ell_x] \times [-\ell_y,\ell_y]$. For the parameter values $r=0.3$, $b=1.2$, $\ell_x=12$ and $\ell_y=6$, multiple local minimizers exist and can be computed as stable steady solutions of the gradient-flow equation
\begin{equation}\label{e:sh-2d-gradient-flow}
\partial_t u = -\Delta^2 u - 2 \Delta u - (1-r)u + \tfrac32 b u^2 - u^3, \qquad u(0,x)=u_0
\end{equation}
with vanishing boundary conditions on $u$ and its derivative normal to the boundary.
Using 100 randomly generated $u_0$, we found the local minimizers plotted in \cref{fig:sh-2d-local-min}.

To seek global minimizers, we discretize~\cref{e:sh-2d-potential} using a conforming FE discretization based on identical triangular elements of reduced Hsieh--Clough--Tocher type~\cite{Clough1965,Meyer2012}.  Our coarsest mesh has 400 identical elements, while the finest one has 3600. On each triangle, $u$ is approximated as a piecewise-cubic function using the values of $u$ and $\nabla u$ at the triangle vertices as the DOF. After the boundary conditions are imposed, the smallest variable cliques are exactly the sets of nine local DOF in each element not intersecting the domain's boundary. We use these cliques in our computations even though they do not satisfy the RIP, so the convergence results from \cref{th:main-result} (suitably extended to problems with second-order derivatives) do not apply. Unfortunately, as in \cref{ss:two-well-problem}, using larger cliques that satisfy the RIP makes computations prohibitively expensive for any reasonable mesh size $h$. To further reduce the computational cost, we imposed bounds $\beta_h$ on the 2-norm of the DOFs in each clique, rather than on individual DOFs (cf. \cref{rem:general-cnstr,rem:general-cnstr-sos}). We considered the two cases $\beta_h = 3/h$ and $\beta_h=4$ uniformly in $h$ for the same reasons discussed in \cref{ss:1d-sh-conforming}.

\begin{figure}
    \includegraphics[trim={2.3cm 0.6cm 2.3cm 0.0cm}, clip, width=\linewidth]{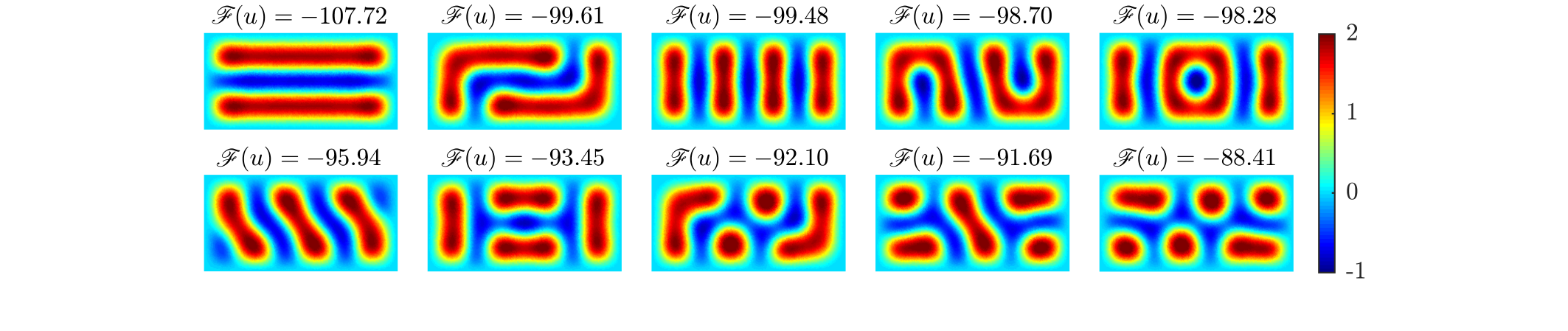}
    \vspace*{-16pt}
    \caption{Local minimizers for problem~\cref{e:sh-2d-potential} on the rectangular domain $\Omega = [-12,12]\times [-6,6]$, obtained by solving the gradient flow equation~\cref{e:sh-2d-gradient-flow}. Corresponding values of the functional $\calF(u)$ minimized in~\cref{e:sh-2d-potential} are also reported.
    \label{fig:sh-2d-local-min}}
    \vspace*{10pt}
    \centering
    \includegraphics[trim={2.3cm 0.6cm 2.3cm 0.2cm}, clip, width=\linewidth]{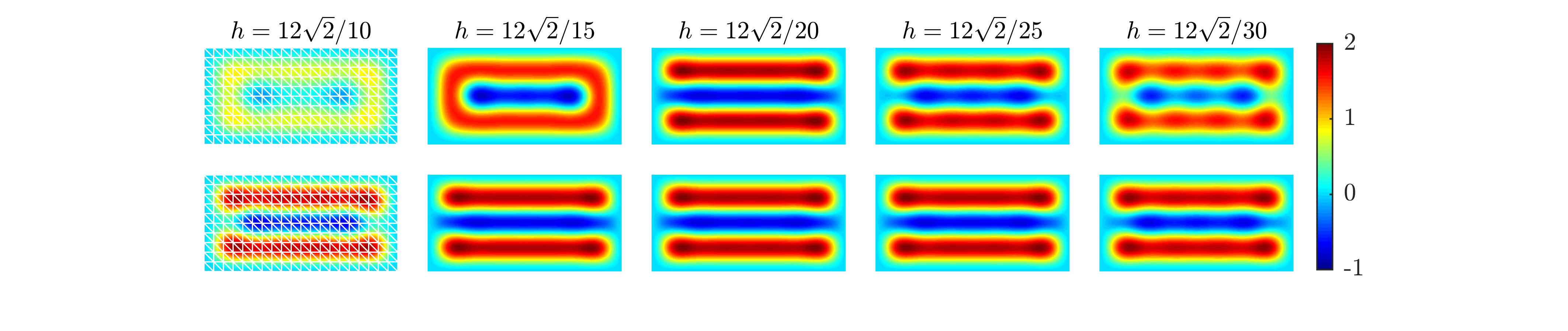}
    \vspace*{-16pt}
    \caption{Approximate minimizers of problem~\cref{e:sh-2d-potential} on $\Omega = [-12,12]\times [-6,6]$, obtained for $\omega = 2$ and decreasing mesh size $h$. The 2-norm of the DOF in each variable clique was bounded by $\beta_h=3/h$ (top row) or by $\beta_h=4$ (bottom row). The mesh is shown only for $h=12\sqrt{2}/10$.
    \label{fig:sh-2d-approx-global-min}}
    \vspace*{10pt}
    \centering
    \includegraphics[trim={2.3cm 0.6cm 2.3cm 0.0cm}, clip, width=\linewidth]{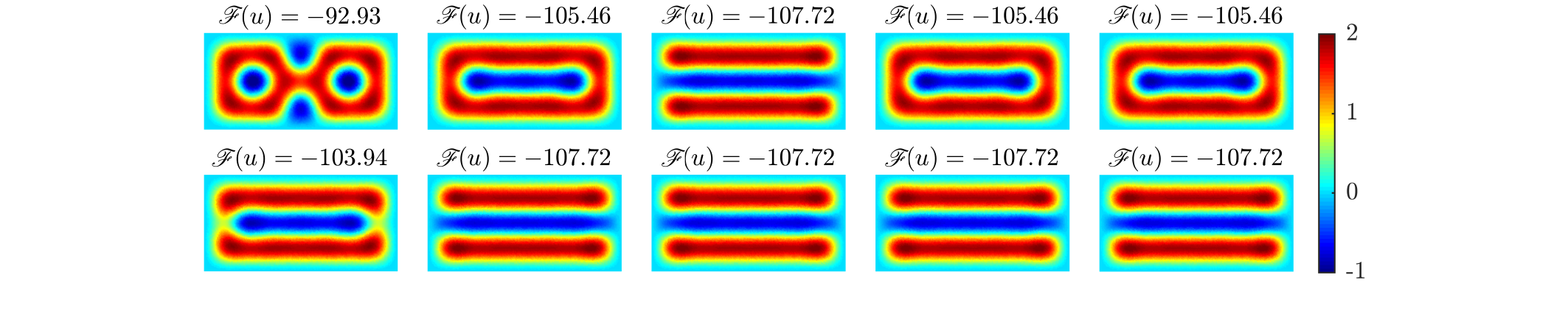}
    \vspace*{-16pt}
    \caption{Local minimizers for problem~\cref{e:sh-2d-potential} on $\Omega = [-12,12]\times [-6,6]$, obtained with Newton iterations and mesh refinement starting from the approximate global minimizers in the corresponding panels of \cref{fig:sh-2d-approx-global-min}. Corresponding values of the functional $\calF(u)$ minimized in~\cref{e:sh-2d-potential} are also reported.
    \label{fig:sh-2d-newton}}
\end{figure}

\Cref{fig:sh-2d-approx-global-min} shows approximate global minimizer recovered from SDP relaxations of order $\omega=2$ for increasingly small mesh size $h$ and both choices of $\beta_h$.  The results strongly suggest that the best local minimizer in \cref{fig:sh-2d-local-min} is globally optimal. As in \cref{ss:1d-sh-conforming}, worse minimizer approximations are obtained when the DOF bounds are too restrictive or when the mesh is too fine and convergence in $\omega$ is further from occurring. Increasing the relaxation order is likely to improve the results for both choices of $\beta_h$, but computations with $\omega=3$ and the general-purpose SDP solver \texttt{MOSEK} were too expensive for our computational resources. Overcoming this scalability problem, perhaps by using a tailored SDP solver, remains an open challenge.

Finally, we use the approximate minimizers from \cref{fig:sh-2d-approx-global-min} to initialize Newton iterations with mesh refinement and obtain the high-resolution local minimizers for \cref{e:sh-2d-potential} shown in \cref{fig:sh-2d-newton}. 
All of these local minimizers have energy values within 14\% of the conjectured global optimum. Moreover, with only one exception, they are better local minimizers that those found with the gradient flow equation \cref{e:sh-2d-gradient-flow} and shown in \cref{fig:sh-2d-local-min}.
Even when our method does not produce good approximations to global minimizers, therefore, it can be used to obtain good local ones.

\section{Conclusions}\label{s:conclusion}
We proposed a `discretize-then-relax' strategy to numerically solve for the global minimizer of integral functionals that satisfy classical coercivity, growth, and quasiconvexity conditions. First, a convergent `bounded' FE discretization scheme is used to approximate the integral minimization problem with a finite-dimensional polynomial optimization problem (POP) over a compact feasible set. Then, sparsity-exploiting moment-SOS relaxations are used to approximate this POP with a convergent hierarchy of SDPs. Similar ideas have been explored in the context of finite-difference schemes \cite{Mevissen2008,Mevissen2009}, but without convergence analysis. Convergence results for FE discretizations that effectively result in POPs over unbounded sets have also been obtained using the framework of $\Gamma$-convergence \cite{Bartels2017sinum,Bartels2017cpa,Grekas2019,Bartels2020}, but no algorithm to compute their global optimizers was given. In contrast, we proved that our numerical strategy produces convergent minimizer approximations for nonconvex integral minimization problems with a unique global minimizer (\cref{th:main-result}). Moreover, with further assumptions on the integrand of the functional being minimized, we obtain a convergent sequence of upper bounds on the global minimum (\cref{th:conv-sos-ub}).

Although in \cref{e:setup:min-problem}  we have restricted ourselves to integral minimization problems constrained only by (homogeneous) Dirichlet boundary conditions, we expect our `discretize-then-relax' to extend with guaranteed convergence to problems where the function $u$ must satisfy differential equations or inequalities. Such problems were already considered in \cite{Mevissen2008,Mevissen2009,Mevissen2011} without convergence proofs, and it remains an open problem to prove rigorous convergence results. Irrespective of what can be proved, the main requirement is that the additional constraints on $u$ be \textit{local}, so that a FE discretization results in a POP with sparsely coupled variables. This is notably not the case for integral constraints, which are typically used to prescribe the mean value or the $L^p$ norm of admissible $u$. For our methods to be applicable, such global constraints should be reformulated as local ones, perhaps by adding auxiliary variables.

We also expect our convergence analysis to extend to non-conforming finite element discretization schemes, such as discontinuous Galerkin methods. Such methods require the addition of `flux terms' that couple neighbouring elements and complicate the convergence analysis (see, e.g., \cite{Bartels2017cpa,Grekas2019}), but are very attractive when conforming discretizations are cumbersome to implement. This is the case, for instance, for the Swift--Hohenberg example in \cref{ss:2d-sh}. In addition, non-conforming discretizations may help reduce the computational cost of the SDP relaxations in \cref{s:sos}. Indeed, recall from \cref{ss:complexity} that this cost is determined by the size of the variable cliques in the discrete problem. While the flux terms force the variable cliques to list the local DOFs of \textit{pairs} of elements, rather than individual ones, the number of DOF per element may be small enough that, overall, the variable cliques remain smaller compared to those of a conforming discretization.

One question that remains open is whether our analysis can be generalized to problems with multiple global minimizers. The main obstruction is that, under conditions (a)--(c) of \cref{th:main-result}, the approximate minimizers $u^*_{h,\omega}$ constructed by our numerical approach are only guaranteed to converge to a \emph{convex combination} of optimizers for the discretized variational problem. Individual optimizers can be extracted from this convex combination if the lower bound $\lambda^*_{h,\omega}$ on the discrete minimum $\calF_h^*$ is exact at a finite $\omega$ and, in addition, the optimal moment and localizing matrices satisfy the so-called flat extension condition (see \cite[\S 3.3]{Lasserre2006}). These conditions hold almost surely for moment-SOS relaxation that do not exploit sparsity~\cite{Nie2014} but it is presently not known if the same is true when sparsity is exploited. One alternative could be to employ minimizer extraction techniques based on finding the global minimizers of the so-called Christoffel--Darboux polynomial (see, e.g, \cite{LasserrePauwels2019,Marx2021}), which would produce approximately optimal DOFs for every clique. However, the Christoffel--Darboux polynomial is typically non-convex, so it is not straightforward to minimize it. It is also not \textit{a priori} clear that the procedure should give DOFs that are consistent among intersecting variable cliques, which is necessary to reconstruct an approximate minimizer $u^*_{h,\omega}$.
Another alternative would be to look for extreme solutions of the SDP~\cref{e:sdp-relaxation}. This would produce sequences of approximate minimizers $u^*_{h,\omega}$ that converge along subsequences in both parameters to global minimizers of the original variational problem. However, looking for minimum-rank solutions is an NP-hard problem~\cite{Lemon2016} and algorithms that seek fixed-rank solutions to SDPs via nonconvex optimization \cite{Burer2002,Burer2003,Burer2005} do not have convergence guarantees unless the chosen rank is sufficiently large \cite{Boumal2016,Waldspurger2020,Boumal2020}.

For the purposes of computational implementation, another undesirable requirement of our analysis is for the variable cliques in the moment-SOS relaxation to satisfy the running intersection property (RIP). The examples in \cref{s:examples} suggest that enforcing this property does not affect the quality of approximate minimizers, but increases the computational complexity to prohibitive levels. However, examples where sparsity-exploiting moment-SOS relaxations without the RIP fail to converge are known (see, e.g., \cite[Example 3.8]{Nie2008}). It remains an open problem to determine if the RIP can or cannot be dropped for POPs obtained through the discretization of integral minimization problems.

Finally, even when the RIP is not enforced, the computational complexity of our `discretize-then-relax' strategy grows rapidly as the relaxation order $\omega$ (hence, the size of the moment and localizing matrices) increases. Thus, despite the theoretical convergence guarantees we proved in this work, advances in computational tools to solve the SDP \cref{e:sdp-relaxation} remain necessary before one can attack complex variational problems arising in physics and engineering. We believe it should be possible to develop more efficient SDP solvers than the general-purpose ones used in this work by exploiting the structure of the SDP \cref{e:sdp-relaxation}. How much can be achieved remains to be seen, but we remark that if our relaxation strategy is used only to provide `educated' initial guesses for traditional methods (e.g., Newton iterations), then one may not need to solve the SDP \cref{e:sdp-relaxation} to high accuracy. The development of specialized first-order solver that can handle very-large instances of~\cref{e:sdp-relaxation}, therefore, could be a fruitful avenue of research.

\bibliographystyle{siamplain}
\bibliography{bib/references-short}
\end{document}